\documentclass[reqno,11pt]{amsart}
\usepackage[latin1]{inputenc}
\usepackage[T1]{fontenc}
\usepackage{lmodern}
\usepackage{a4wide}
\usepackage{amsmath,amsfonts,amsthm,amssymb}
\usepackage{enumitem}
\usepackage{mathtools}
\usepackage{color}
\usepackage{dsfont}
\usepackage{hyperref}
\everymath{\displaystyle}

\theoremstyle{plain}
\newtheorem{theorem}{\bf Theorem}[section]

\newtheorem{lemma}{\bf Lemma}[section]

\newtheorem{definition}{\bf Definition}[section]
\newtheorem{proposition}{\bf Proposition}[section]

\numberwithin{equation}{section}

\title{On a nonlinear eigenvalue problem for generalized Laplacian in Orlicz-Sobolev spaces}

\author[A. Youssfi]{Ahmed Youssfi}
\address{University Sidi Mohamed Ben Abdellah,
	National School of Applied Sciences,
	P.O. Box 72 F\`es-Pricipale, Fez, Morocco}
\email{address:ahmed.youssfi@gmail.com ; ahmed.youssfi@usmba.ac.ma}

\author[M. M. Ould Khatri]{Mohamed Mahmoud Ould Khatri}
\address{University Sidi Mohamed Ben Abdellah,
	National School of Applied Sciences,
	P.O. Box 72 F\`es-Pricipale, Fez, Morocco}
\email{mahmoud.ouldkhatri@usmba.ac.ma}

\date\today
\begin{document}
\maketitle

\begin{abstract}
We consider a nonlinear eigenvalue problem for some elliptic equations governed by general operators including the $p$-Laplacian. The natural framework in which we consider such equations is that of Orlicz-Sobolev spaces. we exhibit two positive constants $\lambda_{0}$ and $\lambda_{1}$ with $\lambda_{0}\leq\lambda_{1}$ such that $\lambda_1$ is an eigenvalue of the problem while any value $\lambda<\lambda_{0}$ cannot be so. By means of Harnack-type inequalities and a strong maximum principle, we prove the isolation of $\lambda_{1}$ on the right side. We emphasize that throughout the paper no $\Delta_2$-condition is needed.
\end{abstract}



\section{Introduction}
\par Let $\Omega$ be an open bounded subset in $\mathbb{R}^{N}$, $N\geq2$, having the segment property. In this paper we investigate the existence and the isolation of an eigenvalue for the following weighted Dirichlet problem
\begin{equation}\label{I}
\left\{
\begin{array}{rll}
-\mbox{div}(\phi(|\nabla u|)\nabla u)&=\lambda \rho(x)\phi(|u|)u&\mbox{ in } \Omega,\\
u&=0&\mbox{ on } \partial\Omega,
\end{array}%
\right.
\end{equation}
where $\phi:(0,\infty)\to(0,\infty)$ is a continuous function, so that defining the function $m(t)=\phi(|t|)t$ we suppose that $m$ is strictly increasing and satisfies $m(t)\to0$ as $t\to0$ and  $m(t)\to\infty$ as $t\to\infty$. The weight function $\rho\in L^{\infty}(\Omega)$ is such that $\rho\geq0$ a.e. in $\Omega$ and $\rho\neq0$ in $\Omega$.
\par If $\phi(t)=|t|^{p-2}$ with $1<p<+\infty$ the problem \eqref{I} is reduced to the eigenvalue problem for the $p$-Laplacian
\begin{equation}\label{I1}
\left\{
\begin{array}{lcl}
-\mbox{div}(|\nabla u|^{p-2}\nabla u)=\lambda \rho(x)|u|^{p-2}u&\mbox{ in }& \Omega,\\
u=0&\mbox{ on }& \partial\Omega,
\end{array}
\right.
\end{equation}
while for $p=2$ and $\rho=1$  it is reduced to the classical eigenvalue problem for the Laplacian
\begin{equation}\label{I2}
\left\{
\begin{array}{lcl}
-\triangle u=\lambda u&\mbox{ in }& \Omega,\\
u=0&\mbox{ on }& \partial\Omega.
\end{array}
\right.
\end{equation}
It is known that the problem \eqref{I2} has a sequence of eigenvalues $0<\lambda_{1}<\lambda_{2}\leq\lambda_{3}\leq\cdots$ such that $\lambda_{n}\to\infty$ as $n\to\infty$. Moreover, the eigenvalues of the problem \eqref{I2} have multiplicities and the first one is simple. Anane \cite{Anane} proved the existence, simplicity and isolation of the first
eigenvalue $\lambda_{1}>0$ of the problem \eqref{I1} assuming some regularity on the boundary $\partial\Omega$.
The simplicity of the first eigenvalue of the problem \eqref{I1} with $\rho=1$ was proved later by Lindqvist \cite{Lindqvist} without any regularity on the domain $\Omega$. For more results on the first eigenvalue of the $p$-Laplacian we refer for example to \cite{Otani,Sakaguchi}.
\par In the general setting  of Orlicz-Sobolev spaces, the following eigenvalue problem
\begin{equation}\label{I3}
\left\{
\begin{array}{lcl}
-\mbox{div}(A(|\nabla u|^{2})\nabla u)=\lambda\psi(u),
&\mbox{ in }& \Omega,\\
u=0&\mbox{ on }& \partial\Omega,
\end{array}%
\right.
\end{equation}
was studied in \cite{Garcia} in the Orlicz-Sobolev space $W^{1}_{0}L_{\Phi}(\Omega)$ where
$\Phi(s)=\displaystyle\int_{0}^{s}A(|t|^{2})t dt$ and $\psi$ is an odd increasing homeomorphism of $\mathbb{R}$ onto $\mathbb{R}$.
In \cite{Garcia} the authors proved the existence of a minimum of the functional $u\to\displaystyle\int_{\Omega}\Phi(|\nabla u|)dx$
which is subject to a constraint and they proved the existence of principal eigenvalues of the problem \eqref{I3} by using
a non-smooth version of the Ljusternik theorem and
by assuming the $\Delta_{2}$-condition on the N-function $\Phi$ and it's complementary $\overline{\Phi}$. Mustonen and Tienari \cite{Mustonen} studied the eigenvalue problem
\begin{equation}\label{I4}
\left\{
\begin{array}{lcl}
-\mbox{div}\Big(\frac{m(|\nabla u|)}{|\nabla u|}\nabla u\Big)=\lambda\rho(x)\frac{m(|u|)}{|u|}u,
&\mbox{ in }& \Omega,\\
u=0&\mbox{ on }& \partial\Omega,
\end{array}
\right.
\end{equation}
in the Orlicz-Sobolev space $W^{1}_{0}L_{M}(\Omega)$, where
$M(s)=\displaystyle\int_{0}^{s}m(t)dt$ with $m(t)=\phi(|t|)t$ and $\rho=1$, without assuming the $\Delta_2$-condition neither on $M$ nor on its conjugate N-function $\overline{M}$. Consequently, the functional $u\to\displaystyle\int_{\Omega}M(|\nabla u|)dx$ is not necessarily continuously differentiable and so classical variational methods can not be applied.
They prove the existence of eigenvalues $\lambda_{r}$ of problem \eqref{I4} with $\rho=1$ and for every $r>0$, by proving the existence of a minimum
of the real valued functional $\displaystyle\int_{\Omega}M(|\nabla u|)dx$ under the constraint $\displaystyle\int_{\Omega}M(u)dx=r$. By the implicit function theorem they proved that every solution of such minimization problem is a weak solution of the problem \eqref{I4}. This result was then extended in \cite{Manasevich} to \eqref{I4} with $\rho\neq1$ and without assuming the $\Delta_{2}$-condition by using a different approach based on a generalized version of Lagrange multiplier rule. The problem \eqref {I} was studied in \cite{Lorca} under the restriction that both the corresponding $N$-function and its complementary function satisfy the $\Delta_{2}$-condition. In reflexive Orlicz-Sobolev spaces, other results related to this topic can be found in \cite{MRR,MR}.
\par In the present paper we define
\begin{equation}\label{lambda0}
\lambda_{0}=\inf_{u\in W^{1}_{0}L_{M}(\Omega)\setminus\{0\}}
\frac{\int_{\Omega}\phi(|\nabla u|)|\nabla u|^{2}dx}{\int_{\Omega}\rho(x)\phi(|u|)|u|^{2}dx}
\end{equation}
and
\begin{equation}\label{lambda1}
\lambda_{1}=\inf\Big\{\int_{\Omega}M(|\nabla u|)dx\;\Big|\; u\in W^{1}_{0}L_{M}(\Omega),\;\int_{\Omega}\rho(x)M(|u|)dx=1\Big\}.
\end{equation}
\par In the particular case where $\phi(t)=|t|^{p-2}$, $1<p<+\infty$, we obtain $\lambda_0=\lambda_1$ and so $\lambda_0=\lambda_1$ is
the first isolated and simple eigenvalue of the problem \eqref{I1} (see \cite{Anane}).
\par However, in the non reflexive Orlicz-Sobolev structure the situation is more complicated since we can not expect that $\lambda_0=\lambda_1$.
Precisely, we can not assert whether $\lambda_0=\lambda_1$ or $\lambda_0<\lambda_1$. We think that this is an open problem and we expect that the answer strongly depends on the $N$-function $M$. If $\lambda_0<\lambda_1$, another open problem is to seek whether $\lambda_1$ is the smallest eigenvalue of problem (1). In other words to investigate the existence of eigenvalues of problem (1) in the interval $[\lambda_0,\lambda_1)$. Nonetheless, we show that $\lambda_0\leq\lambda_{1}$ and that any value $\lambda<\lambda_{0}$ can not be an eigenvalue of the problem \eqref{I}. Following the lines of \cite{Manasevich}, we also show that $\lambda_{1}$ is an eigenvalue of problem \eqref{I} associated to an eigenfunction $u$ which is a weak solution of \eqref{I} (see Definition \ref{def1} below).
It is in our purpose in this paper to prove that $\lambda_{1}$ is isolated from the right-hand side. To do so, we first prove some Harnack-type inequalities that enable us to show that $u$ is H\"older continuous and then by a strong maximum principle we show that $u$ has a constant sign. Besides, we prove that any eigenfunction associated to another eigenvalue than $\lambda_{1}$ necessarily changes its sign. This allows us to prove that $\lambda_1$ is isolated from the right hand side.
\par Let $\Omega$ be an open subset in $\mathbb{R}^N$ and let $M(t)=\int_{0}^{|t|}m(s)ds$, $m(t)=\phi(|t|)t$. The natural framework in which we consider the problem \eqref{I} is the Orlicz-Sobolev space defined by
$$
W^{1}L_{M}(\Omega)=\Big\{u\in L_{M}(\Omega):\partial_{i}u:=\frac{\partial u}{\partial x_{i}}\in L_{M}(\Omega),
i=1,\cdots,N\Big\}.
$$
where $L_{M}(\Omega)$ stands for the Orlicz space defined as follows
$$
L_{M}(\Omega)=\Big\{u:\Omega\rightarrow\mathbb{R}\mbox{ measurable }:
\displaystyle\int_{\Omega}M\Big(\frac{|u(x)|}{\lambda}\Big)dx<\infty\mbox{ for some }\lambda>0\Big\}.
$$
The spaces $L_{M}(\Omega)$ and $W^{1}L_{M}(\Omega)$ are Banach spaces under their respective norms
$$
\|u\|_{M}=\inf\Big\{\lambda>0:\int_{\Omega}M\Big(\frac{|u(x)|}{\lambda}\Big)dx\leq1\Big\}\mbox{ and }
\|u\|_{1,M}=\|u\|_{M}+\|\nabla u\|_{M}.
$$
The closure in $L_{M}$ of the set of bounded measurable functions with compact support in $\overline{\Omega}$ is denoted by $E_{M}(\Omega)$. The complementary function $\overline{M}$ of the $N$-function $M$ is defined by
$$
\overline{M}(x,s)=\sup_{t\geq0}\{st-M(x,t)\}.
$$
Observe that by the convexity of $M$ follows the inequality
\begin{equation}\label{prop.b}
\|u\|_{M}\leq\int_{\Omega}M(|u(x)|)dx+1\mbox{ for all }u\in L_{M}(\Omega).
\end{equation}
Denote by $W^{1}_{0}L_{M}(\Omega)$ the closure of $C_{0}^{\infty}(\Omega)$ in $W^{1}L_{M}(\Omega)$ with respect to the
weak* topology $\sigma(\Pi L_{M},\Pi E_{\overline{M}})$. It is known that if $\Omega$ has the segment property, then the four spaces
$$
(W^{1}_{0}L_{M}(\Omega),W^{1}_{0}E_{M}(\Omega);W^{-1}L_{\overline{M}}(\Omega),W^{-1}E_{\overline{M}}(\Omega))
$$
form a complementary system (see \cite{Gossez}).
If $\Omega$ is bounded in $\mathbb{R}^N$ then by the Poincar\'e inequality \cite[lemma 5.7]{Gossez}, $\|u\|_{1,M}$ and $\|\nabla u\|_{M}$ are
equivalent norms in $W^{1}_{0}L_{M}(\Omega)$.
\par Let $J:D(J)\to\mathbb{R}\cup\{+\infty\}$ and
$B:W^{1}_{0}L_{M}(\Omega)\to\mathbb{R}$ are the two functionals defined by
\begin{equation}\label{J}
J(u)=\int_{\Omega}M(|\nabla u|)dx
\end{equation}
and
\begin{equation}\label{B}
B(u)=\int_{\Omega}\rho(x)M(|u|)dx,
\end{equation}
respectively. The functional $J$ takes values in $\mathbb{R}\cup\{+\infty\}$. Since $W^{1}_{0}L_{M}(\Omega)\subset E_{M}(\Omega)$ (see \cite{Manasevich}), then the functional $B$ is real valued on $W^{1}_{0}L_{M}(\Omega)$.
Set
$$
K=\{u\in W^{1}_{0}L_{M}(\Omega):B(u)=1\}.
$$
In general, the functional $J$ is not finite nor of class $\mathcal{C}^{1}$ (see \cite{Mustonen} p. 158).
\section{Main results}
We will show that $\lambda_1$ given by relation \eqref{lambda1} is an eigenvalue of the problem \eqref{I} and isolated from the right hand side, while any $\lambda<\lambda_0$ is not an eigenvalue
of \eqref{I}. In the sequel we assume that $\Omega$ is a bounded domain (unless otherwise stated) in $\mathbb{R}^{N}$ having the segment property. 
\begin{definition}\label{def1}
	A function $u$ is said to be a weak solution of \eqref{I} associated with
	$\lambda\in\mathbb{R}$ if
	\begin{equation}\label{4}
	\begin{cases}
	u\in W^{1}_{0}L_{M}(\Omega), m(|\nabla u|)\in L_{\overline{M}}(\Omega)\\
	\int_{\Omega}\phi(|\nabla u|)\nabla u\cdot\nabla\psi dx=
	\lambda\int_{\Omega}\rho(x)\phi(|u|)u\psi dx, \mbox{ for all } \psi\in W^{1}_{0}L_{M}(\Omega)
	\end{cases}
	\end{equation}
\end{definition}
In this definition, both of the two integrals in (\ref{4}) make sense. Indeed, for all $u\in W^{1}_{0}L_{M}(\Omega)$
since $m(|\nabla u|)=\phi(|\nabla u|)|\nabla u|\in L_{\overline{M}}(\Omega)$, the first term is well defined. From the Young inequality and the integral representation of $M$, we easily get $\overline{M}(m(u))\leq um(u)\leq M(2u)$. So that since $u\in E_{M}(\Omega)$ the integral on the right-hand side also makes sense.
\begin{definition}
	We said that $\lambda$ is an eigenvalue of the problem (\ref{I}), if there exists a function $v\neq0$ belonging to $W^{1}_{0}L_{M}(\Omega)$ such that $(\lambda,v)$ satisfy (\ref{4}). The function $v$ will be called an eigenfunction associated with the eigenvalue $\lambda$.
\end{definition}
\subsection{Existence result}
\par We start with the next result that can be found in \cite[Lemma 3.2]{Mustonen}. For the convenience of the reader we give here a slightly different proof.
\begin{lemma}\label{lemma1}
	Let $J$ and $B$ be defined by \eqref{J} and \eqref{B}. Then \\
	(i) $B$ is $\sigma(\Pi L_{M},\Pi E_{\overline{M}})$ continuous,\\
	(ii) $J$ is $\sigma(\Pi L_{M},\Pi E_{\overline{M}})$ lower semi-continuous.
\end{lemma}
\begin{proof}
	(i) Let $u_{n}\to u$ for $\sigma(\Pi L_{M},\Pi E_{\overline{M}})$ in $W^{1}_{0}L_{M}(\Omega)$.
	By the compact embedding $W^{1}_{0}L_{M}(\Omega)\hookrightarrow
	E_{M}(\Omega)$, $u_{n}\to u$ in
	$E_{M}(\Omega)$ in norm. Hence $M(2(u_{n}-u))\to 0$ in $L^{1}(\Omega)$.
	By the dominated convergence theorem, there exists a subsequence of $\{u_{n}\}$ still denoted
	by $\{u_{n}\}$ with
	$u_{n}\to u$ a.e. in $\Omega$ and there exists $h\in L^{1}(\Omega)$ such that
	$$
	M(2(u_{n}-u))\leq h(x)\mbox{ a.e. in }\Omega
	$$
	for a subsequence. Therefore,
	$$
	|u_{n}|\leq |u|+\frac{1}{2}M^{-1}(h),
	$$
	so
	$$
	M(u_{n})\leq \frac{1}{2}M(2u)+\frac{1}{2}h(x)
	$$
	and since $\rho\geq0$ for a.e. in $\Omega$, then
	$$
	\rho(x)M(u_{n})\leq \frac{1}{2}\rho(x)M(2u)+\frac{1}{2}\rho(x)h(x)\in L^{1}(\Omega).
	$$
	Thus, the assertion (i) follows from the dominated convergence theorem.
	\par To show (ii) we assume that $u_{n}\to u$ for $\sigma(\Pi L_{M},\Pi E_{\overline{M}})$
	in $W^{1}_{0}L_{M}(\Omega)$, that is
	$$
	\int_{\Omega}u_{n}v dx\to\displaystyle\int_{\Omega}uv dx\mbox{ and }\displaystyle\int_{\Omega}\partial_{i}u_{n}v dx
	\to\int_{\Omega}\partial_{i}uv dx,
	$$
	for all $v\in E_{\overline{M}}$. This holds, in particular, for all $v\in L^{\infty}(\Omega)$. Hence,
	\begin{equation}\label{faibl.cv1}
	\partial_{i}u_{n}\to \partial_{i}u\mbox{ and }u_{n}\to u
	\mbox{ in }L^{1}(\Omega)\mbox{ for }\sigma(L^{1},L^{\infty}).
	\end{equation}
	Since the embedding $W^{1}_{0}L_{M}(\Omega)\hookrightarrow L^{1}(\Omega)$ is compact,
	then $\{u_{n}\}$ is relatively compact in $L^{1}(\Omega)$. By passing to a subsequence, $u_{n}\to v$ strongly in $L^{1}(\Omega).$
	In view of \eqref{faibl.cv1}, $v=u$ and $u_{n}\to u$ strongly in $L^{1}(\Omega)$.
	Passing once more to a subsequence, we obtain that $u_{n}\to u$ almost everywhere on $\Omega$. Since $\zeta\mapsto M(|\zeta|)$
	is convex for $\zeta\in\mathbb{R}^{N}$, we can use \cite[Theorem 2.1, Chapter 8]{Ekeland}, to obtain
	$$
	J(u)=\int_{\Omega}M(|\nabla u|)dx\leq\lim\inf\int_{\Omega}M(|\nabla u_{n}|)dx=
	\lim\inf J(u_{n}).
	$$
\end{proof}
The first result of this paper is given by the following theorem.
\begin{theorem}\label{thm1}
	The infimum in \eqref{lambda1} is
	achieved at some function $u\in K$ which is a weak solution of \eqref{I} and thus $u$ is an eigenfunction associated to the eigenvalue $\lambda_1$.
	Furthermore, $\lambda_0\leq\lambda_1$ and each
	$\lambda<\lambda_0$ is not an eigenvalue of problem \eqref{I}.
\end{theorem}
\begin{proof}
	We split the proof of Theorem \ref{thm1} into three steps.\\
	\textbf{Step 1 }: We show that the infimum in \eqref{lambda1} is achieved at some $u\in K$.
	By \eqref{prop.b} we have
	$$
	J(u)=\int_{\Omega}M(|\nabla u|)dx\geq\|\nabla u\|_{M}-1.
	$$
	So, $J$ is coercive. Let $\{u_{n}\}\subset W^{1}_{0}L_{M}(\Omega)$ be a minimizing sequence, i.e. $u_{n}\in K$ and $u_{n}\to\inf_{v\in K}J(v)$.
	The coercivity of $J$
	implies that $\{u_{n}\}$ is bounded in $W^{1}_{0}L_{M}(\Omega)$ which is in the dual of a separable Banach
	space. By the Banach-Alaoglu-Bourbaki theorem, there exists $u\in W^{1}_{0}L_{M}(\Omega)$ such that
	for a subsequence still indexed by $n$, $u_n\to u$ for $\sigma(\Pi L_{M},\Pi E_{\overline{M}})$
	in $W^{1}_{0}L_{M}(\Omega)$. As a consequence of Lemma \ref{lemma1} the set $K$ is closed with respect to the topology $\sigma(\Pi L_{M},\Pi E_{\overline{M}})$ in $W^{1}_{0}L_{M}(\Omega)$. Thus, $u\in K$. Since $J$ is $\sigma(\Pi L_{M},\Pi E_{\overline{M}})$ lower semi-continuous, it follows
	$$
	J(u)\leq\lim\inf J(u_{n})=\inf_{v\in K}J(v),
	$$
	which shows that $u$ is a solution of \eqref{lambda1}.\\
	\textbf{Step 2 }: The function $u\in K$ found in Step 1 is such that $m(|\nabla u|)\in L_{\overline{M}}(\Omega)$ and satisfies \eqref{4}. This was already proved in \cite[Theorem 4.2]{Manasevich}.\\
	\textbf{Step 3 }: Let $\lambda_{0}$ be given by (\ref{lambda0}). Any value $\lambda<\lambda_{0}$ cannot be an eigenvalue of problem (\ref{I}). Indeed, suppose by contradiction that there exists a value
	$\lambda\in(0,\lambda_{0})$ which is an eigenvalue of
	problem \eqref{I}. It follows that there exists $u_{\lambda}\in W^{1}_{0}L_{M}(\Omega)\setminus\{0\}$ such that
	$$
	\displaystyle\int_{\Omega}\phi(|\nabla u_{\lambda}|)\nabla u_{\lambda}\cdot\nabla vdx=
	\lambda\displaystyle\int_{\Omega}\rho(x)\phi(|u_{\lambda}|)u_{\lambda}vdx\mbox{ for all }v\in W^{1}_{0}L_{M}(\Omega).
	$$
	Thus, in particular for $v=u_{\lambda}$ we can write
	$$
	\int_{\Omega}\phi(|\nabla u_{\lambda}|)|\nabla u_{\lambda}|^{2}dx=
	\lambda\int_{\Omega}\rho(x)\phi(|u_{\lambda}|)|u_{\lambda}|^{2}dx.
	$$
	The fact that $u_{\lambda}\in W^{1}_{0}L_{M}(\Omega)\setminus\{0\}$ ensures that
	$\displaystyle\int_{\Omega}\rho(x)\phi(|u_{\lambda}|)|u_{\lambda}|^{2}dx>0$.
	By the definition of $\lambda_{0}$, we obtain
	$$
	\begin{array}{clc}
	\int_{\Omega}\phi(|\nabla u_{\lambda}|)|\nabla u_{\lambda}|^{2}dx&\geq
	\lambda_{0}\int_{\Omega}\rho(x)\phi(|u_{\lambda}|)|u_{\lambda}|^{2}dx&\\
	&>\lambda\int_{\Omega}\rho(x)\phi(|u_{\lambda}|)|u_{\lambda}|^{2}dx&\\
	&=\int_{\Omega}\phi(|\nabla u_{\lambda}|)|\nabla u_{\lambda}|^{2}dx.
	\end{array}
	$$
	Which yields a contradiction. Therefore, we conclude that $\lambda_0\leq\lambda_1$. The proof of Theorem \ref{thm1} is now complete.
\end{proof}
\subsection{Isolation result}
In this subsection we first show a maximum principle which enables us to prove that any eigenfunction associated to $\lambda_1$ has a constant sign in $\Omega$. This property is then used to prove that $\lambda_1$ is isolated from the right-hand side.
\par Let $w$ be an eigenfunction of problem (\ref{I}) associated to the eigenvalue $\lambda_1$. Since $|w|\in K$ it
follows that $|w|$ achieves also the infimum in \eqref{lambda1}, which implies that $|w|$ is also an eigenfunction associated to $\lambda_1$. So we can assume that $w$ is non-negative, that is
$$
w(x)\geq 0 \mbox{ for }x\in\Omega.
$$
Since by Theorem \ref{bounded} the eigenfunction $w$ is bounded, we set 
$$
0\leq\delta:=\sup_\Omega w<+\infty.
$$
For $t\in(0,\delta)$ the function $f(t)=\phi(t)t=m(t)>0$ is continuous and strictly increasing. Let $F(s)=\displaystyle\int_{0}^{s}f(t)dt$.
We assume that 
\begin{equation}\label{H}
\int_{0}^{\delta}\frac{ds}{H^{-1}(M(s))}=+\infty,
\end{equation}
where $H$ is the function defined for all $t\geq0$ by
$$
H(t)=tm(t)-M(t)=\overline{M}(m(t)).
$$
The assumption (\ref{H}) is known to be a necessary condition for the strong maximum principle to hold (see \cite{Pucci} and the references therein). Hereafter, under (\ref{H}) we can compare $w$ to a suitable function given by \cite[Lemma 2]{Pucci}.
\par The proof of the strong maximum principle will be given after proving the following two Lemmas.
\begin{lemma}\label{lem.eq} 
	Denote by $B(y,R)$ an open ball in $\Omega$ of radius $R$ and centered at $y\in\Omega$ and consider the annulus
	$$
	E_R=\Big\{x\in B(y,R):\frac{R}{2}\leq|y-x|<R\Big\}.
	$$
	Assume that \eqref{H} holds. Then there exists a function $v\in C^{1}$ with
	$0<v<\delta$, $v^{\prime}<0$ in $E_R$ and $w\geq v$ on $\partial E_R$. Moreover, $v$ satisfies 
	\begin{equation}\label{sol.distr}
	-\int_{\Omega}\phi(|\nabla v|)\nabla v\cdot\nabla \psi dx\leq\int_{\Omega}f(v)\psi dx,
	\end{equation}
	for every $\psi\in W^{1}_{0}L_M(\Omega)$ and $\psi\leq0$.
\end{lemma}
\begin{proof}
	Let $r=|y-x|$ for $x\in\overline{E}_R$. The function $v(x)=v(r)$ given by \cite[Lemma 2]{Pucci} satisfies for every positive numbers $k$, $l$, and for $\epsilon\in(0,\delta)$ 
	$$
	[m(|v^\prime|)]^\prime+\frac{k}{r}m(|v^\prime|)+lf(v)\leq0,
	$$
	$0<v<\epsilon<\delta$, $v^{\prime}<0$ in $E_R$ and $v(x)=0$ if $|y-x|=R$.
	In addition, for $x\in E_{R}$ with $|y-x|=\frac{R}{2}$ we have
	$v(x)<\epsilon<\inf_{\{x:|y-x|=\frac{R}{2}\}}w(x)<\delta$.
	Hence, follows $w\geq v$ on $\partial E_R$. Moreover, by the radial symmetric expression of $\mbox{div}(\phi(|\nabla v|)\nabla v)$, we have
	$$
	\begin{array}{clc}
	\mbox{div}(\phi(|\nabla v|)\nabla v)-f(v)
	&=-[m(|v^\prime|)]^\prime-\frac{(N-1)}{r}m(|v^\prime|)-f(v)\geq0,
	\end{array}
	$$
	where we recall that $v^\prime<0$ and use \cite[Lemma 2]{Pucci}. Multiplying the above inequality by $\psi\in W^{1}_{0}L_M(\Omega)$ with $\psi\leq0$ and then integrating over $\Omega$ we obtain (\ref{sol.distr}). The proof is achieved.
\end{proof}
\begin{lemma}[Weak comparison principle]\label{comparison} 
	Assume that \eqref{H} holds. Let $v$ be the $\mathcal{C}^1$-function given by Lemma \ref{lem.eq} with $0<v<\delta$ in $\Omega$ and $w\geq v$ on $\partial\Omega$. Then $w\geq v$ in $\Omega$.
\end{lemma}
\begin{proof}
	Let $h=w-v$ in $\Omega$. Assume by contradiction that there exists $x_1\in\Omega$ such that $h(x_1)<0$. Fix $\epsilon>0$ so small that $h(x_1)+\epsilon<0$. By Theorem \ref{Holder} (see Appendix) the function $w$ is continuous in $\Omega$, then so is the function $h$. Since $h\geq0$ on $\partial\Omega$,  the support $\Omega_0$ of the function $h_\epsilon=\min\{h+\epsilon,0\}$ is a compact subset in $\Omega$. By Theorem \ref{thm.Adx1} (see Appendix), the function $h_\epsilon$ belongs to $W^{1}_{0}L_{M}(\Omega)$. Taking it as a test function in \eqref{4} and \eqref{sol.distr} it yields
	$$
	\int_{\Omega_0}\phi(|\nabla w|)\nabla w\cdot(\nabla w-\nabla v )dx=\lambda_{1}\int_{\Omega_0}\rho(x)\phi(|w|)wh_\epsilon dx
	$$
	and
	$$
	-\int_{\Omega_0}\phi(|\nabla v|)\nabla v\cdot(\nabla w-\nabla v )dx\leq\int_{\Omega_0}\phi(|v|)vh_\epsilon dx.
	$$
	Summing up the two formulations, we obtain
	\begin{equation}\label{eq1}
	\int_{\Omega_0}[\phi(|\nabla w|)\nabla w-\phi(|\nabla v|)\nabla v]\cdot(\nabla w-\nabla v)dx\leq\int_{\Omega_0}(\lambda_{1}\rho(x)m(w)+m(v))h_\epsilon dx.
	\end{equation}
	The left-hand side of \eqref{eq1} is positive due to Lemma \ref{monotone} (given in Appendix), while the right-hand side of \eqref{eq1} is non positive, since $h_\epsilon<0$ in $\Omega_0$. Therefore,
	$$
	\int_{\Omega_0}[\phi(|\nabla w|)\nabla w-\phi(|\nabla v|)\nabla v]\cdot(\nabla w-\nabla v)dx=0
	$$
	implying $\nabla h_\epsilon=0$ and so $h+\epsilon>0$ which contradicts the fact that $h(x_1)+\epsilon<0$.
\end{proof}
\par Now we can prove our strong maximum principle.
\begin{theorem}[Strong maximum principle]\label{maximum} 
	Assume that \eqref{H} holds. Then, if $w$ is a non-negative eigenfunction associated to $\lambda_1$, then $w>0$ in $\Omega$.
\end{theorem}
\begin{proof}
	Let $B(y,R)$ be an open ball of $\Omega$ of radius $R$ and centered at a fixed arbitrary $y\in\Omega$. We shall prove that $w(x)>0$ for all $x\in B(y,R)$. Let $v$ be the $\mathcal{C}^1$-function given by Lemma \ref{lem.eq} with $w\geq v$ on $\partial E_R$ where
	$$
	E_R=\Big\{x\in B(y,R):\frac{R}{2}\leq|y-x|<R\Big\}.
	$$
	Applying Lemma \ref{comparison} we get $w\geq v>0$ in $E_R$. For $|y-x|<\frac{R}{2}$ we consider
	$$
	E_\frac{R}{2}=\Big\{x\in B(y,R):\frac{R}{4}\leq|y-x|<\frac{R}{2}\Big\}.
	$$
	We can us similar arguments as in the proof of Lemma \ref{lem.eq} to obtain that there is $v\in C^{1}$ in $E_\frac{R}{2}$, with $v>0$ in $E_\frac{R}{2}$ and $w\geq v$ on $\partial E_\frac{R}{2}$. Applying again Lemma \ref{comparison} we obtain $w\geq v>0$ in $E_\frac{R}{2}$.
	So, by the same way we can conclude that $w(x)>0$ for any $x\in B(y,R)$.
\end{proof}
\par Now we are ready to prove that the associated eigenfunction of $\lambda_1$ has necessarily a constant sign in $\Omega$.
\begin{proposition}\label{prop1} 
	Assume that \eqref{H} holds. Then, every eigenfunction $u$ associated to the eigenvalue $\lambda_1$ has constant sign in $\Omega$, that is, either $u>0$ in $\Omega$ or $u<0$ in $\Omega$.
\end{proposition}
\begin{proof}
	Let $u$ be an eigenfunction associated to the eigenvalue $\lambda_1$. Then $u$ achieves the infimum in \eqref{lambda1}. Since $|u|\in K$ it follows that $|u|$ achieves also the infimum in \eqref{lambda1}, which implies that $|u|$ is also an eigenfunction associated to $\lambda_1$. Therefore, applying Theorem \ref{maximum} with $|u|$ instead of $w$, we obtain $|u|>0$ for all $x\in\Omega$ and since $u$ is continuous (see Theorem \ref{Holder} in Appendix), then, either $u>0$ or $u<0$ in $\Omega$.
\end{proof}
Before proving the isolation of $\lambda_1$, we shall prove that every eigenfunction associated to another eigenvalue $\lambda>\lambda_{1}$ changes in force its sign in $\Omega$. Denote by $|E|$ the Lebesgue measure of a subset $E$ of $\Omega$.
\begin{proposition}
	Assume that \eqref{H} holds. If $v\in W^{1}_{0}L_{M}(\Omega)$ is an eigenfunction associated to an eigenvalue $\lambda>\lambda_{1}$. Then $v^{+}\ncong0$ and $v^{-}\ncong0$ in $\Omega$. Moreover, if we set $\Omega^{+}=\{x\in\Omega:v(x)>0\}$ and $\Omega^{-}=\{x\in\Omega:v(x)<0\}$, then
	\begin{equation}\label{estimate1}
	\min\{|\Omega^{+}|,|\Omega^{-}|\}\geq \min\bigg\{\frac{1}{\overline{M}\Big(\frac{dc}{\min\{a,1\}}\Big)},
	\frac{1}{\overline{M}\Big(\frac{dc}{\min\{b,1\}}\Big)}\bigg\}
	\end{equation}
	where $a=\int_{\Omega}v^{+}(x)dx$, $b=\int_{\Omega}v^{-}(x)dx$,
	$c=c(\lambda,|\Omega|,\|v\|_{\infty},\|\rho\|_{\infty})$ and $d$ is the
	constant in the Poincar\'e norm inequality (see \cite[Lemma 5.7]{Gossez}).
\end{proposition}
\begin{proof}
	By contradiction, we assume that there exists an eigenfunction $v$ associated to 
	$\lambda>\lambda_{1}$ such that $v>0$. The case $v<0$ being completely
	analogous so we omit it . Let $u>0$ be an eigenfunction associated to $\lambda_{1}$. Let $\Omega_{0}$ be a compact subset of $\Omega$ and define the two functions
	$$
	\eta_{1}(x)=\left\{
	\begin{array}{lcl}
	u(x)-v(x)+\sup_{\Omega}v&\mbox{ if }&x\in\Omega_{0}\\
	0&\mbox{ if }&x\notin\Omega_{0}
	\end{array}%
	\right.
	$$
	and
	$$
	\eta_{2}(x)=\left\{
	\begin{array}{lcl}
	v(x)-u(x)-\sup_{\Omega}v&\mbox{ if }&x\in\Omega_{0}\\
	0&\mbox{ if }&x\notin\Omega_{0}.
	\end{array}
	\right.
	$$
	Pointing out that $v$ is bounded (Theorem \ref{bounded} in Appendix), the two functions $\eta_{1}$ and $\eta_{2}$ are admissible test functions in (\ref{4}) (see Theorem \ref{thm.Adx1} in Appendix). Thus, we have
	$$
	\int_{\Omega}\phi(|\nabla u|)\nabla u\cdot\nabla\eta_{1}dx=\lambda_{1}\int_{\Omega}\rho(x)\phi(|u|)u\eta_{1}dx
	$$
	and
	$$
	\int_{\Omega}\phi(|\nabla v|)\nabla v\cdot\nabla\eta_{2}dx=\lambda\int_{\Omega}\rho(x)\phi(|v|)v\eta_{2}dx.
	$$
	By summing up and using Lemma \ref{monotone} (in Appendix), we get
	$$
	\begin{array}{lcl}
	0&\leq\int_{\Omega}[\phi(|\nabla u|)\nabla u-\phi(|\nabla v|)\nabla v]\cdot(\nabla u-\nabla v)dx\\
	&=\int_{\Omega}\rho(x)\Big(\lambda_{1}m(u)-\lambda m(v)\Big)(u-v+\sup_{\Omega}v)dx.
	\end{array}
	$$
	We claim that  
	$$
	\lambda_{1}m(u)\leq\lambda m(v).
	$$ 
	Indeed, suppose that $\lambda_{1}m(u)>\lambda m(v)$ and let us  define the two admissible test functions 
	$$
	\eta_{3}(x)=\left\{
	\begin{array}{lcl}
	u(x)-v(x)-\sup_{\Omega}u&\mbox{ if }&x\in\Omega_{0},\\
	0&\mbox{ if }&x\notin\Omega_{0}
	\end{array}
	\right.
	$$
	and
	$$
	\eta_{4}(x)=\left\{
	\begin{array}{lcl}
	v(x)-u(x)+\sup_{\Omega}u&\mbox{ if }&x\in\Omega_{0},\\
	0&\mbox{ if }&x\notin\Omega_{0}.
	\end{array}
	\right.
	$$
	As above, inserting $\eta_{3}$ and $\eta_{4}$ in (\ref{4}) and then summing up we obtain
	$$
	\begin{array}{lcl}
	0&\leq&\int_{\Omega}[\phi(|\nabla u|)\nabla u-\phi(|\nabla v|)\nabla v]\cdot(\nabla u-\nabla v)dx\\
	&=&\int_{\Omega}\rho(x)[\lambda_{1}m(u)-\lambda m(v)](u-v-\sup_{\Omega}u)dx\leq0,
	\end{array}
	$$
	implying by Lemma \ref{monotone} that $v=u$, but such an equality can not occur since $\lambda>\lambda_1$ which proves our claim. Finally, we conclude that the function $v$ can not have a constant sign in $\Omega$. 
	\par Next we prove the estimate \eqref{estimate1}.
	According to the above $v^{+}>0$ and $v^{-}>0$. Choosing $v^{+}\in W^{1}_{0}L_{M}(\Omega)$ as a test function in \eqref{4}, we get
	$$
	\int_{\Omega}m(|\nabla v^{+}|)|\nabla v^{+}|dx=\lambda\int_{\Omega}\rho(x)m(v^{+})v^{+} dx.
	$$
	Since $M(t)\leq m(t)t\leq M(2t)$ for $t\geq0$, we obtain
	$$
	\int_{\Omega}M(|\nabla v^{+}|)dx\leq\lambda\|\rho(\cdot)\|_{\infty}\int_{\Omega}M(2v^{+})dx.
	$$
	We already know that by Theorem \ref{bounded} (in Appendix) the function $v$ is bounded, then we get
	\begin{equation}\label{3.inequal}
	\int_{\Omega}M(|\nabla v^{+}|)dx\leq\lambda\|\rho(\cdot)\|_{\infty}M(2\|v\|_{\infty})|\Omega|.
	\end{equation}
	So, \eqref{prop.b} and \eqref{3.inequal} imply that there exists a positive constant $c$, such that
	\begin{equation}\label{ineq.1}
	\|\nabla v^{+}\|_{M}\leq c.
	\end{equation}
	On the other hand, by the H\"{o}lder inequality \cite{Kras} and the Poincar\'e type inequality \cite[Lemma 5.7]{Gossez}, we have
	$$
	\int_{\Omega}v^{+}(x)dx\leq\|\chi_{\Omega^{+}}\|_{\overline{M}}\|v^{+}\|_{M}\leq d\|\chi_{\Omega^{+}}\|_{\overline{M}}
	\|\nabla v^{+}\|_{M},
	$$
	$d$ being the constant in Poincar\'e type inequality. Hence, using \eqref{ineq.1} to get
	\begin{equation}\label{ineq.3}
	\int_{\Omega}v^{+}(x)dx\leq cd\|\chi_{\Omega^{+}}\|_{\overline{M}}.
	\end{equation}
	We have to distinguish two cases, the case $\int_{\Omega}v^{+}(x)dx>1$ and $\int_{\Omega}v^{+}(x)dx\leq1$.\\
	\underline{Case 1} : Assume that
	$$
	\int_{\Omega}v^{+}(x)dx>1.
	$$
	Thus, by \eqref{ineq.3} we have
	\begin{equation}\label{cse1}
	\frac{1}{dc}\leq \|\chi_{\Omega^{+}}\|_{\overline{M}}.
	\end{equation}
	\underline{Case 2} : Assume that
	$$
	\int_{\Omega}v^{+}(x)dx\leq1.
	$$
	Recall that by Theorem \ref{Holder} (in Appendix) the function $v^{+}$ is continuous and as $v^{+}>0$ in $\Omega$ then $\int_{\Omega}v^{+}(x)dx>0$. Therefore, by using \eqref{ineq.3} we obtain
	\begin{equation}\label{cse2}
	\frac{a}{dc}\leq \|\chi_{\Omega^{+}}\|_{\overline{M}},
	\end{equation}
	where $a=\int_{\Omega}v^{+}(x)dx$. So, by \eqref{cse1} and \eqref{cse2}, we get
	$$
	\frac{\min\{a,1\}}{dc}\leq \|\chi_{\Omega^{+}}\|_{\overline{M}},
	$$
	where $\|\chi_{\Omega^{+}}\|_{\overline{M}}=\frac{1}{\overline{M}^{-1}\Big(\frac{1}{|\Omega^{+}|}\Big)}$ (see \cite[page 79]{Kras}).
	Hence,
	$$
	|\Omega^{+}|\geq\frac{1}{\overline{M}\Big(\frac{dc}{\min\{a,1\}}\Big)}.
	$$
	Such an estimation with $v^{-}$ can be obtained following exactly the same lines above. Then follows the inequality \eqref{estimate1}.
\end{proof}
\par Finally, we prove that the eigenvalue $\lambda_1$ given by the relation \eqref{lambda1} is isolated from the right-hand side.
\begin{theorem}
	Assume that \eqref{H} holds. Then, the eigenvalue $\lambda_{1}$ is isolated from the right-hand side, that is, there exists $\delta>0$ such that in the
	interval $(\lambda_{1},\lambda_{1}+\delta)$ there are no  eigenvalues.
\end{theorem}
\begin{proof}
	Assume by contradiction that there exists a non-increasing sequence $\{\mu_{n}\}_n$ of eigenvalues of \eqref{I} with $\mu_{n}>\lambda_{1}$ and $\mu_{n}\to\lambda_{1}$. Let $u_{n}$ be an associated eigenfunction to $\mu_{n}$ and let
	$$
	\overline{\Omega_{n}^{+}}=\overline{\{x\in\Omega:u_{n}>0\}}\mbox{ and }\overline{\Omega_{n}^{-}}=\overline{\{x\in\Omega:u_{n}<0\}}.
	$$
	By \eqref{estimate1}, there exists $c_n>0$ such that
	\begin{equation}\label{estimate2}
	\min\{|\overline{\Omega_{n}^{+}}|,|\overline{\Omega_{n}^{-}}|\}\geq c_n.
	\end{equation}
	Since $b_{n}:=\int_{\Omega}\rho(x)M(|u_{n}(x)|)dx>0$ we define
	\begin{equation}\label{vn}
	v_{n}(x)=\left\{
	\begin{array}{lcl}
	M^{-1}\Big(\frac{M(u_{n}(x))}{b_{n}}\Big)&\mbox{ if }&x\in\overline{\Omega_{n}^{+}},\\
	-M^{-1}\Big(\frac{M(-u_{n}(x))}{b_{n}}\Big)&\mbox{ if }&x\in\overline{\Omega_{n}^{-}}.\\
	\end{array}
	\right.
	\end{equation}
	On the other hand, we have
	$$
	|\nabla v_{n}|\leq \Big|(M^{-1})^{\prime}\Big(\frac{M(|u_{n}|)}{b_{n}}\Big)\Big|\frac{m(|u_{n}|)|\nabla u_{n}|}{b_{n}}
	\chi_{\overline{\Omega_{n}^{+}}\cup\overline{\Omega_{n}^{-}}},
	$$
	since $u_n$ is continuous, then there exists $d_{n}>0$ such that $\inf_{x\in\overline{\Omega_{n}^{+}}\cup\overline{\Omega_{n}^{-}}}|u_{n}(x)|\geq d_{n}$.
	Let $b=\min\{b_{n}\}$ and $d=\min\{d_{n}\}$. Being  $\{u_n\}$ uniformly bounded (Theorem \ref{bounded} in Appendix), there exists a constant $c_\infty>0$, not depending on $n$, such that
	\begin{equation}\label{un}
	\|u_{n}\|_{\infty}\leq c_\infty, \mbox{ for all }n\in\mathbb{N}.
	\end{equation}
	Using the fact that $(M^{-1})^{\prime}(\cdot)$ is decreasing, we get
	\begin{equation}\label{vn1}
	\begin{array}{lcl}
	|\nabla v_{n}|\leq\Big|(M^{-1})^{\prime}\Big(\frac{M(d)}{\|\rho\|_{\infty}
		M(c_\infty)|\Omega|}\Big)\Big|\frac{m(c_\infty)}{b}|\nabla u_{n}|
	=C_0|\nabla u_{n}|,
	\end{array}
	\end{equation}
	where $C_0=\Big|(M^{-1})^{\prime}\Big(\frac{M(d)}{\|\rho\|_{\infty}
		M(c_\infty)|\Omega|}\Big)\Big|\frac{m(c_\infty)}{b}$.
	On the other hand, taking $u_{n}$ as test function in \eqref{4} and using \eqref{un} and the inequality $M(t)\leq m(t)t$ for $t>0$, one has
	$$
	\int_{\Omega}M(|\nabla u_{n}|)dx\leq\mu_{n}\|\rho\|_{\infty}m(c_\infty)c_\infty|\Omega|.
	$$
	Since $\mu_{n}$ converges to $\lambda_{1}$, there exists a constant $C_1>0$, such that
	\begin{equation}\label{un1}
	\int_{\Omega}M(|\nabla u_{n}|)dx\leq C_1.
	\end{equation}
	Therefore, by \eqref{vn1} and \eqref{un1}
	we obtain that $\{v_{n}\}$ is uniformly bounded in $W^{1}_{0}L_{M}(\Omega)$.
	Alaoglu's theorem ensures the existence of a function $v\in W^{1}_{0}L_{M}(\Omega)$ and a subsequence of $v_n$, still indexed by $n$, such that $v_{n}\rightharpoonup v$ for
	$\sigma(\Pi L_{M},\Pi E_{\overline{M}})$.
	By \eqref{vn}, $v_{n}\in K$ and since $B$ is $\sigma(\Pi L_{M},\Pi E_{\overline{M}})$ continuous (see Lemma \ref{lemma1}), then
	$$
	\int_{\Omega}\rho(x)M(|v(x)|)dx=B(v)=\lim_{n\to\infty} B(v_{n})=1.
	$$
	Therefore, $v\in K$. Since by Lemma \ref{lemma1} the functional $J$ is $\sigma(\Pi L_{M},\Pi E_{\overline{M}})$ lower semi-continuous, we get
	$$
	J(v)=\int_{\Omega}M(|\nabla v|)dx\leq\lim\inf J(v_{n})=\inf_{w\in K}J(w).
	$$
	So that $v$ is an eigenfunction associated to $\lambda_{1}$. Applying Proposition \ref{prop1}, we have either $v>0$ or $v<0$ in $\Omega$.
	Assume that $v<0$  in $\Omega$ with $v^{-}\ncong0$.
	By Egorov's Theorem, $v_{n}$ converges uniformly to $v$ except on a subset of $\Omega$ of null Lebesgue measure.
	Thus, $v_{n}\leq0$ a.e. in $\Omega$ with $v_{n}^{-}\ncong0$ outside a subset of $\Omega$ of null Lebesgue measure,
	which implies that
	$$
	|\overline{\Omega_{n}^{+}}|=0,
	$$
	which is a contradiction with the estimation \eqref{estimate2}.
\end{proof}
\section*{Appendix}
We prove here some important lemmas that are necessary for the accomplishment of the proofs of the above results.
\begin{lemma}\label{monotone}
	Let $\xi$ and $\eta$ be vectors in $\mathbb{R}^N$. Then
	$$
	[\phi(|\xi|)\xi-\phi(|\eta|)\eta]\cdot(\xi-\eta)>0,
	\mbox{ whenever } \xi\neq\eta.
	$$
\end{lemma}
\begin{proof}
	Since $\phi(t)>0$ when $t>0$ and $\xi\cdot\eta\leq|\xi|\cdot|\eta|$, there follows by a direct calculation
	$$
	[\phi(|\xi|)\xi-\phi(|\eta|)\eta]\cdot(\xi-\eta)\geq[m(|\xi|)-m(|\eta|)]\cdot(|\xi|-|\eta|)
	$$
	and the conclusion comes from the strict monotonicity of $m$.
\end{proof}
\par The following result can be found in \cite[Lemma 9.5]{Brezis} in the case of Sobolev spaces.
\begin{theorem}\label{thm.Adx1} Let $A$ be an $N$-function (cf. \cite{Adams}).
	If $u\in W^{1}L_{A}(\Omega)$ has a compact support in an open $\Omega$ having the segment property, then $u\in W^{1}_{0}L_{A}(\Omega)$.
\end{theorem}
\begin{proof}
	Let $u\in W^{1}L_{A}(\Omega)$. We fix a compact set $\Omega^{\prime}\subset\Omega$ such that $supp\; u\subset\Omega^{\prime}$ and we denote by $\bar{u}$ the extension by zero of $u$ to the whole of $\mathbb{R}^N$.	Let $J$ be the Friedrichs mollifier kernel defined on $\mathbb{R}^{N}$ by
	$$
	\rho(x)=ke^{-\frac{1}{1-\|x\|^{2}}}\mbox{ if }\|x\|<1\mbox{ and }0\mbox{ if }\|x\|\geq1,
	$$
	where $k>0$ is such that $\int_{\mathbb{R}^{N}}\rho(x)dx=1$. For $\epsilon>0$, we define
	$\rho_{n}(x)=n^{N}J(nx)$.
	By \cite{Gossez82}, there exists $\lambda>0$ large enough such that $A\Big(\frac{|u(x)|}{\lambda}\Big)\in L^1(\Omega)$, $A\Big(\frac{|\partial u/\partial x_i(x)|}{\lambda}\Big)\in L^1(\Omega)$, $i\in\{1,\cdots,N\}$, and
	$$
	\int_{\mathbb{R}^{N}}A\Big(\frac{|\rho_{n}\ast\bar{u}(x)-\bar{u}(x)|}
	{\lambda}\Big)dx\to 0 \mbox{ as } n\to+\infty
	$$	
	and hence	
	\begin{equation}\label{omega}
	\int_{\Omega}A\Big(\frac{|\rho_{n}\ast\bar{u}(x)-u(x)|}
	{\lambda}\Big)dx\to 0 \mbox{ as } n\to+\infty.
	\end{equation}	
	Choosing $n$ large enough so that $0<\frac{1}{n}<dist(\Omega^{\prime},\partial\Omega)$ one has $\rho_{n}\ast\bar{u}(x)=\rho_{n}\ast u(x)$ for every $x\in\Omega^{\prime}$. Hence, $\partial(\rho_{n}\ast\bar{u})/\partial x_i= \rho_{n}\ast(\partial u/\partial x_i)$ on $\Omega^{\prime}$ for every $i\in\{1,\cdots,N\}$. As $\partial u/\partial x_i\in L_{A}(\Omega^{\prime})$ we have
	$$
	\partial(\rho_{n}\ast\bar{u})/\partial x_i \in L_{A}(\Omega^{\prime}).
	$$ 
	Therefore, 
	\begin{equation}\label{omegapr}
	\int_{\Omega^{\prime}}A\Big(\frac{|\partial(\rho_{n}\ast\bar{u})/\partial x_i(x)-\partial u/\partial x_i(x)|}
	{\lambda}\Big)dx\to 0 \mbox{ as } n\to+\infty.
	\end{equation}
	Observe that the functions $w_n=\rho_{n}\ast\bar{u}$ do not necessary lie in $\mathcal{C}_{0}^{\infty}(\Omega)$. Let $\eta\in\mathcal{C}_{0}^{\infty}(\mathbb{R}^{N})$ such that $0<\eta<1$, $\eta(x)=1$ for all $x$ with $\|x\|\leq1$, $\eta(x)=0$ for all $x$ with $\|x\|\geq2$ and $|\nabla\eta|\leq2$. Let further $\eta_n(x)=\eta\Big(\frac{x}{n}\Big)$ for $x\in\mathbb{R}^{N}$. We claim that the functions $v_n=\eta_n w_n$ belong to $\mathcal{C}_{0}^{\infty}(\mathbb{R}^{N})$ and satisfy
	\begin{equation}\label{omega1}
	\int_{\Omega}A\Big(\frac{|v_{n}(x)-u(x)|}{4\lambda}\Big)dx\to 0 \mbox{ as } n\to+\infty
	\end{equation}
	and 
	\begin{equation}\label{omegapr1}
	\int_{\Omega^{\prime}}A\Big(\frac{|\partial v_{n}/\partial x_i (x)-\partial u/\partial x_i(x)|}{12\lambda}\Big)dx\to 0 \mbox{ as } n\to+\infty
	\end{equation} 
	Indeed, by (\ref{omega}) there exist a subsequence of $\{w_n\}$ still indexed by $n$ and a function $h_1\in L^{1}(\Omega)$ such that
	$$
	w_n\to u \mbox{ a.e. in } \Omega
	$$
	and
	\begin{equation}\label{h1}
	|w_n(x)|\leq |u(x)|+\lambda A^{-1}(h_1)(x); \mbox{ for all }x\in\Omega
	\end{equation}
	which together with the convexity of $A$ yield
	$$
	A\Big(\frac{|v_{n}(x)-u(x)|}{4\lambda}\Big)\leq \frac{1}{2}A\Big(\frac{|u(x)|}{\lambda}\Big)+\frac{1}{4}h_1(x)
	$$
	Being the functions $A\Big(\frac{|u(x)|}{\lambda}\Big)\in L^1(\Omega)$ and $h_1\in L^{1}(\Omega)$, the sequence $\Big\{A\Big(\frac{|v_{n}-u|}{4\lambda}\Big)\Big\}_n$ is equi-integrable on $\Omega$ and since $\{v_n\}$ converges to $u$ a.e. in $\Omega$, we obtain (\ref{omega1}) by applying Vitali's theorem.\\
	By (\ref{omegapr}) there exists a subsequence, relabeled again by $n$, and a function $h_2\in L^1(\Omega^{\prime})$ such that
	$$
	\partial w_{n}/\partial x_i\to \partial u/\partial x_i \mbox{ a.e. in } \Omega^{\prime}
	$$
	and
	\begin{equation}\label{h2}
	|\partial w_{n}/\partial x_i (x)|\leq |\partial u/\partial x_i (x)|+h_2(x), \mbox{ for all }x\in \Omega^{\prime}.
	\end{equation}
	Therefore, using (\ref{h1}) and (\ref{h2}) for all $x\in\Omega^{\prime}$ we arrive at
	$$
	\begin{array}{lll}
	A\Big(\frac{|\partial v_{n}/\partial x_i (x)-\partial u/\partial x_i(x)|}{12\lambda}\Big)\\
	\leq \frac{1}{6}\bigg(A\Big(\frac{|u(x)|}{\lambda}\Big)+A\Big(\frac{|\partial u/\partial x_i(x)|}{\lambda}\Big)+h_1(x)+\frac{1}{2}h_2(x)\bigg)
	\end{array}
	$$
	and by Vitali's theorem we obtain (\ref{omegapr1}).\\
	Finally, let $K\subset\Omega^{\prime}$	be a compact set such that $supp(u)\subset K$. There exists a cut-off function $\zeta\in \mathcal{C}_{0}^{\infty}(\Omega^{\prime})$ satisfying $\zeta=1$ on $K$. Denoting $u_n=\zeta v_{n}$, we can deduce from (\ref{omega1}) and (\ref{omegapr1})
	$$
	\int_{\Omega}A\Big(\frac{|u_n(x)-\zeta u(x)|}{12\lambda}\Big)dx\to 0 \mbox{ as } n\to+\infty
	$$
	and
	$$
	\int_{\Omega}A\Big(\frac{|\partial u_{n}/\partial x_i (x)-\partial (\zeta u)/\partial x_i(x)|}{12\lambda}\Big)dx\to 0 \mbox{ as } n\to+\infty.
	$$
	Consequently, the sequence $\{ u_n\}\subset\mathcal{C}_{0}^{\infty}(\Omega)$ converges modularly to $\zeta u=u$ in
	$W^1L_{A}(\Omega)$ and in force for the weak topology $\sigma(\Pi L_{A},\Pi L_{\overline{A}})$ (see \cite[Lemma 6]{Gossez82}) which in turn imply the convergence with respect to the weak$^\ast$ topology $\sigma(\Pi L_{A},\Pi E_{\overline{A}})$. Thus, $u\in W^1_0L_{A}(\Omega)$.
\end{proof}
\begin{theorem}\label{bounded}
	For any weak solution $u\in W^{1}_{0}L_{M}(\Omega)$ of \eqref{I}  associated with $\lambda>0$, there exists a constant $c_{\infty}>0$, not depending on $u$, such that 
	$$
	\|u\|_{L^{\infty}(\Omega)}\leq c_{\infty}.
	$$
\end{theorem}
\begin{proof}
	For $k>0$ we define the set $A_k=\{x\in\Omega:|u(x)|>k\}$ and the two truncation functions $T_{k}(s)=\max(-k,\min(s,k))$ and $G_{k}(s)=s-T_{k}(s)$. By H\"older's inequality we get
	$$
	\begin{array}{clc}
	\displaystyle\int_{A_k}|G_{k}(u(x))|dx&\leq|A_k|^{\frac{1}{N}}\Big(\displaystyle\int_{A_k}|G_{k}(u(x))|^{\frac{N}{N-1}}dx\Big)^{\frac{N-1}{N}}\\
	&\leq C(N)|A_k|^{\frac{1}{N}}\displaystyle\int_{A_k}|\nabla u|dx,
	\end{array}
	$$
	where $C(N)$ is the constant in the embedding $W^{1,1}_0(A_k)\hookrightarrow L^{\frac{N}{N-1}}(A_k)$. We shall estimate the integral $\int_{A_k}|\nabla u|dx$; to this aim we distinguish two cases : the case $m(|\nabla u|)|\nabla u|<\lambda_{1}\|\rho\|_{\infty}$ and $m(|\nabla u|)|\nabla u|\geq\lambda_{1}\|\rho\|_{\infty}$, where $\lambda_{1}$ is defined in (\ref{lambda1}).
	\par \underline{Case 1} : Assume that
	\begin{equation}\label{case1}
	m(|\nabla u|)|\nabla u|<\lambda_{1}\|\rho\|_{\infty}.
	\end{equation}
	Let $k_0>0$ be fixed and let $k>k_0$. Using \eqref{case1} we can write
	$$
	\begin{array}{clc}
	\displaystyle\int_{A_k}|\nabla u|dx&\leq\displaystyle\int_{A_k\cap\{|\nabla u|\leq1\}}|\nabla u|dx+
	\displaystyle\int_{A_k\cap\{|\nabla u|>1\}}|\nabla u|dx\\
	&\leq |A_k|+\frac{1}{m(1)}\displaystyle\int_{A_k}m(|\nabla u|)|\nabla u|dx\\
	&\leq\Big(1+\frac{\lambda_{1}\|\rho\|_{\infty}}{m(1)}\Big)|A_k|.
	\end{array}
	$$
	Thus,
	\begin{equation}\label{stam1}
	\int_{A_k}|G_{k}(u(x))|dx\leq C(N)\Big(1+\frac{\lambda_{1}\|\rho\|_{\infty}}{m(1)}\Big)|A_k|^{\frac{1}{N}+1}.
	\end{equation}
	\par \underline{Case 2} : Assume now that
	\begin{equation}\label{case2}
	m(|\nabla u|)|\nabla u|\geq\lambda_{1}\|\rho\|_{\infty}.
	\end{equation}
	Since $u\in W^{1}_0L_M(\Omega)$ is a weak solution of problem (\ref{I}), we have
	\begin{equation}\label{bounded2}
	\int_\Omega\phi(|\nabla u|)\nabla u\cdot\nabla v dx=\lambda\int_\Omega\rho(x)\phi(|u|)uv dx,
	\end{equation}
	for all $v\in W^{1}_0L_M(\Omega)$. For $s$, $t$, $k>0$ we define the following function $v=\frac{\lambda_{1}}{\lambda}\exp\Big(\frac{\lambda}{\lambda_{1}}M(u^+)\Big)T_s(G_k(T_t(u^+)))$.
	From \cite[Lemma 2]{Gossez86} we know that $v$ is an admissible test function in \eqref{bounded2}. Taking it so it yields 
	$$
	\begin{array}{clc}
	&\frac{\lambda}{\lambda_{1}}\int_{\{u>0\}} m(|\nabla u|)|\nabla u|m(u^+)vdx&\\
	&+\frac{\lambda_{1}}{\lambda}\int_{\{k<T_t(u^+)\leq k+s\}}\phi(|\nabla u|)\nabla u\cdot\nabla T_t(u^+)\exp\Big(\frac{\lambda}{\lambda_{1}}M(u^+)\Big)dx&\\
	&-\lambda\int_{\{u>0\}}\rho(x)\phi(|u|)uvdx=0.
	\end{array}
	$$
	Since we integrate on the set $\{u>0\}$, by \eqref{case2} we have
	$$
	\lambda_1\rho(x)\leq m(|\nabla u|)|\nabla u|
	$$
	and so we obtain
	$$
	\begin{array}{lll}\displaystyle
	\frac{\lambda}{\lambda_{1}}\int_{\{u>0\}}\Big(m(|\nabla u|)|\nabla u|-\lambda_1\rho(x)\Big) m(u^+)vdx \geq0.
	\end{array}
	$$
	Therefore, we have
	$$
	\int_{\{k<T_t(u^+)\leq k+s\}}\phi(|\nabla u|)\nabla u\cdot\nabla T_t(u^+)\exp\Big(\frac{\lambda}{\lambda_{1}}M(u^+)\Big)dx=0
	$$
	and since $\exp\Big(\frac{\lambda}{\lambda_{1}}M(u^+)\Big)\geq1$ we get
	$$
	\int_{\{k<T_t(u^+)\leq k+s\}}\phi(|\nabla u|)\nabla u\cdot\nabla T_t(u^+)dx=0.
	$$
	Pointing out that 
	$$
	\begin{array}{clc}
	\int_{\{k<T_t(u^+)\leq k+s\}}\phi(|\nabla u|)\nabla u\cdot\nabla T_t(u^+)dx
	=\int_{\{k<u\leq k+s\}\cap\{0<u<t\}}\phi(|\nabla u|)\nabla u\cdot\nabla udx,
	\end{array}
	$$
	we can apply the monotone convergence theorem as $t\to+\infty$ obtaining
	$$
	\begin{array}{clc}\displaystyle
	\int_{\{k<u\leq k+s\}}\phi(|\nabla u|)\nabla u\cdot\nabla udx
	&=\displaystyle\lim_{t\to\infty}\displaystyle\int_{\{k<T_t(u^+)\leq k+s\}}\phi(|\nabla u|)\nabla u\cdot\nabla T_t(u^+)dx\\
	&=0.
	\end{array}
	$$
	Applying again the monotone convergence theorem as $s\to+\infty$ we get
	\begin{equation}\label{1}
	\begin{array}{clc}
	\int_{\{u>k\}}\phi(|\nabla u|)\nabla u\cdot\nabla udx
	=\lim_{s\to\infty}\displaystyle\int_{\{k<u\leq k+s\}}\phi(|\nabla u|)\nabla u\cdot\nabla udx
	=0.
	\end{array}
	\end{equation}
	In the same way, inserting the function $v=-\frac{\lambda_{1}}{\lambda}\exp\Big(\frac{\lambda}{\lambda_{1}}M(u^-)\Big)T_s(G_k(T_t(u^-)))$ that belongs to $W^{1}_{0}L_{M}(\Omega)$ as a test function in \eqref{bounded2} we obtain
	$$
	\begin{array}{clc}
	&-\frac{\lambda}{\lambda_1}\int_{\{u<0\}}\phi(|\nabla u|)\nabla u\cdot\nabla um(u^-)vdx&\\
	&-\frac{\lambda_1}{\lambda}\int_{\{-k-s\leq T_t(u^-)<-k\}}\phi(|\nabla u|)\nabla u\cdot\nabla T_t(u^-)\exp\Big(\frac{\lambda}{\lambda_1}M\Big(u^-\Big)\Big)dx&\\
	&=\lambda\int_{\{u<0\}}\rho(x)\phi(|u|)uvdx.
	\end{array}
	$$
	Then we can write
	$$
	\begin{array}{clc}
	&-\frac{\lambda}{\lambda_1}\int_{\{u<0\}} m(|\nabla u|)|\nabla u|m(|u|)vdx&\\
	&-\frac{\lambda_1}{\lambda}\int_{\{-k-s\leq T_t(u^-)<-k\}}\phi(|\nabla u|)\nabla u\cdot\nabla T_t(u^-)\exp\Big(\frac{\lambda}{\lambda_1}M\Big(u^-\Big)\Big)dx&\\
	&=-\lambda\int_{\{u<0\}}\rho(x)m(|u|)vdx.
	\end{array}
	$$
	Gathering the first term in the left-hand side and the term in the right-hand side of the above equality, we get 	
	$$
	\begin{array}{clc}
	&-\frac{\lambda}{\lambda_1}\int_{\{u<0\}}\Big(m(|\nabla u|)|\nabla u|-\lambda_1\rho(x)\Big)m(|u|)vdx&\\
	&-\frac{\lambda_1}{\lambda}\int_{\{-k-s\leq T_t(u^-)<-k\}}\phi(|\nabla u|)\nabla u\cdot\nabla T_t(u^-)\exp\Big(\frac{\lambda}{\lambda_1}M\Big(u^-\Big)\Big)dx=0.
	\end{array}
	$$	
	Here again since we have $\lambda_1\rho(x)\leq m(|\nabla u|)|\nabla u|,$
	we obtain
	$$
	-\frac{\lambda_1}{\lambda}\int_{\{-k-s\leq T_t(u^-)<-k\}}\phi(|\nabla u|)\nabla u\cdot\nabla T_t(u^-)\exp\Big(\frac{\lambda}{\lambda_1}M\Big(u^-\Big)\Big)dx\leq0,
	$$
	that is
	$$
	\frac{\lambda_1}{\lambda}\int_{\{-k-s\leq T_t(u^-)<-k\}\cap\{u>-t\}}\phi(|\nabla u|)\nabla u\cdot\nabla u\exp\Big(\frac{\lambda}{\lambda_1}M\Big(u^-\Big)\Big)dx\leq0.
	$$
	As $\exp\Big(\frac{\lambda}{\lambda_1}M\Big(u^-\Big)\Big)\geq1$ we get
	$$
	\int_{\{-k-s\leq T_t(u^-)<-k\}\cap\{u>-t\}}\phi(|\nabla u|)\nabla u\cdot\nabla udx=0.
	$$
	As above applying the monotone convergence theorem successively as $t\to+\infty$ and then $s\to+\infty$, we arrive at
	\begin{equation}\label{2}
	\int_{\{u<-k\}}\phi(|\nabla u|)\nabla u\cdot\nabla udx=0.
	\end{equation}
	Thus, from (\ref{1}) and (\ref{2}) and since $m(t)=\phi(|t|)t$ we conclude that  
	\begin{equation}\label{bounded4}
	\int_{A_k}m(|\nabla u|)|\nabla u|dx=0.
	\end{equation}
	On the other hand, by the  monotonicity of the function $m^{-1}$ and by \eqref{bounded4}, we can write
	$$
	\begin{array}{clc}
	\int_{A_k}|\nabla u|dx&=\int_{A_k\cap\{m(|\nabla u|)<1\}}|\nabla u|dx+
	\int_{A_k\cap\{m(|\nabla u|)\geq1\}}|\nabla u|dx\\
	&\leq m^{-1}(1)|A_k|+\int_{A_k}m(|\nabla u|)|\nabla u|dx\\
	&=m^{-1}(1)|A_k|.
	\end{array}
	$$
	Hence,
	\begin{equation}\label{stam2}
	\int_{A_k}|G_{k}(u(x))|dx\leq C(N)m^{-1}(1)|A_k|^{\frac{1}{N}+1}.
	\end{equation}
	Finally, we note that the two inequalities (\ref{stam1}) and (\ref{stam2}) yield exactly the starting point of Stampacchia's $L^\infty$-regularity proof (see \cite{stamp65}).
	In Fact, in any case we always have 
	\begin{equation}\label{stam22}
	\int_{A_k}|G_{k}(u(x))|dx\leq \eta|A_k|^{\frac{1}{N}+1},
	\end{equation}
	where $\eta:=C(N)\Big(1+m^{-1}(1)+\frac{\lambda_1\|\rho\|_{\infty}}{m(1)}\Big)$. Let $h>k>0$. It is easy to see that $A_{h}\subset A_{k}$ and $|G_{k}(u)|\geq h-k$ on $A_{h}$. Thus, we have
	$$
	(h-k)|A_{h}|\leq \eta|A_k|^{\frac{1}{N}+1}.
	$$
	The nonincreasing function $\psi$ defined by $\psi(k)=|A_{k}|$ satisfies
	$$
	\psi(h)\leq\frac{\eta}{(h-k)}\psi(k)^{\frac{1}{N}+1}.
	$$
	Applying the first item of \cite[Lemma 4.1]{stamp65} we obtain
	$$
	\psi(c_{\infty})=0 \mbox{ where } c_{\infty}=C(N)\Big(1+m^{-1}(1)+\frac{\lambda_1\|\rho\|_{\infty}}{m(1)}\Big)2^{N+1}|\Omega|^{\frac{1}{N}},
	$$
	which yields
	$$
	\|u\|_{L^{\infty}(\Omega)}\leq c_{\infty}=C(N)\Big(1+m^{-1}(1)+\frac{\lambda_1\|\rho\|_{\infty}}{m(1)}\Big)2^{N+1}|\Omega|^{\frac{1}{N}}.
	$$
\end{proof}
\begin{lemma}\label{lem.Adx2}
	Let $\Omega$ be an open bounded subset in $\mathbb{R}^{N}$. Let $B_{R}\subset\Omega$ be an open ball of radius $0<R\leq1$. Suppose that $g$ is a non-negative function such that $g^\alpha\in L^{\infty}(B_{R})$, where $|\alpha|\geq1$. Assume that
	\begin{equation}\label{eq1.Adx2}
	\Big(\int_{B_{R}}g^{\alpha qk}dx\Big)^{\frac{1}{k}}\leq C\int_{B_{R}}g^{\alpha q}dx,
	\end{equation}
	where $q,k>1$ and $C$ is a positive constant. Then for any $p>0$ there exists a positive constant $c$ such that
	$$
	\sup_{B_{R}}g^{\alpha}\leq \frac{c}{R^{\frac{k}{(k-1)p}}}\Big(\int_{B_{R}}g^{\alpha p}dx\Big)^{\frac{1}{p}}.
	$$
\end{lemma}
\begin{proof}
	Let $q=pk^{\nu}$ where $\nu$ is a non-negative integer. Then using   \eqref{eq1.Adx2} and the fact that $R\leq 1$ we can have
	$$
	\Big(\int_{B_{R}}g^{\alpha pk^{\nu+1}}dx\Big)^{\frac{1}{pk^{\nu+1}}}\leq 
	\Big(\frac{C}{R}\Big)^{\displaystyle\frac{1}{pk^{\nu}}}
	\Big(\int_{B_{R}}g^{\alpha pk^{\nu}}dx\Big)^{\frac{1}{pk^{\nu}}}.
	$$
	An iteration of this inequality with respect to $\nu$ yields
	\begin{equation}\label{eq2.Adx2}
	\|g^{\alpha}\|_{L^{pk^{\nu+1}}(B_{R})}\leq
	\Big(\frac{C}{R}\Big)^{\displaystyle\frac{1}{p}\sum_{i=0}^{\nu}
		\frac{1}{k^i}}
	\Big(\int_{B_{R}}g^{\alpha p}dx\Big)^{\frac{1}{p}}.
	\end{equation}
	For $\beta\geq1$, we consider $\nu$ large enough such that $pk^{\nu+1}>\beta$. Then, there exists a constant $c_0$
	such that
	$$
	\|g^{\alpha}\|_{L^{\beta}(B_{R})}\leq c_0\|g^{\alpha}\|_{L^{pk^{\nu+1}}(B_{R})}.
	$$
	Since the series in \eqref{eq2.Adx2} are convergent and $g^{\alpha}\in L^{\infty}(B_{R})$,
	Theorem 2.14 in \cite{Adams} implies that
	$$
	\sup_{B_{R}}g^{\alpha}\leq \frac{c}{R^{\frac{k}{(k-1)p}}}\Big(\int_{B_{R}}g^{\alpha p}dx\Big)^{\frac{1}{p}}.
	$$
\end{proof}
As we need to get a H\"older estimate for weak solutions of \eqref{I}, we use the previous lemma to prove Harnack-type inequalities. To do this, we define for a bounded weak solution $u\in W^{1}_{0}L_{M}(\Omega)$ of \eqref{I} the two functions $v=u-\inf_{B_{r}}u$ and $w=\sup_{B_{r}}u-u$.
We start by proving the following two lemmas.
\begin{lemma}
	Let $B_{r}\subset\Omega$ be an open ball of radius $0<r\leq1$. Then for every $p>0$, there exists a positive constant $C$, depending on $p$, such that
	\begin{equation}\label{sup1}
	\sup_{B_{\frac{r}{2}}}v\leq C\bigg(\Big(r^{-N}\int_{B_{r}}v^{p}dx\Big)
	^{\frac{1}{p}}+r\bigg),
	\end{equation}
	where $B_{\frac{r}{2}}$ is the ball of radius $r/2$ concentric with $B_{r}$.
\end{lemma}
\begin{proof}
	Since $u$ is a weak solution of problem \eqref{I} then $v$ satisfies the weak formulation  
	\begin{equation}\label{4.1}
	\int_{\Omega}\phi(|\nabla v|)\nabla v\cdot\nabla\psi dx=\lambda\int_{\Omega}\rho(x)\phi(|v+\inf_{B_r}u|)(v+\inf_{B_r}u)\psi dx,
	\end{equation}
	for every $\psi\in W^{1}_{0}L_{M}(\Omega)$. Let $\Omega_{0}$ be a compact subset of $\Omega$ such that $B_{\frac{r}{2}}\subset\Omega_{0}\subset B_{r}$.
	Let $q>1$ and let $\psi$ be the function defined by
	$$
	\psi(x)=\left\{
	\begin{array}{lcl}
	M(\bar{v}(x))^{q-1}\bar{v}(x)&\mbox{ if }&x\in\Omega_{0},\\
	0&\mbox{ if }&x\notin \Omega_{0}
	\end{array}
	\right.
	$$
	where $\bar{v}=v+r$. Observe that on $\Omega_0$
	$$
	\nabla\psi=M(\bar{v})^{q-1}\nabla\bar{v}+(q-1)
	M(\bar{v})^{q-2}m(\bar{v})\bar{v}\nabla\bar{v}
	$$
	and thus by Theorem \ref{thm.Adx1} we have $\psi\in W^{1}_{0}L_{M}(\Omega)$.
	So that $\psi$ is an admissible test function in \eqref{4.1}. Taking it so it yields
	$$
	\begin{array}{clc}
	\int_{B_{r}}M(\bar{v})^{q-1}m(|\nabla\bar{v}|)|\nabla\bar{v}|dx
	&+(q-1)\int_{B_{r}}M(\bar{v})^{q-2}m(\bar{v})\bar{v}m(|\nabla\bar{v}|)|\nabla\bar{v}|dx&\\
	&=\lambda\int_{B_{r}}\rho(x)M(\bar{v})^{q-1}\bar{v}m(v+\inf_{B_r}u)dx.
	\end{array}
	$$
	Since $\bar{v}m(\bar{v})\geq M(\bar{v})$ and $v+\inf_{B_r}u\leq\bar{v}+\|u\|_{\infty}$, we get
	\begin{equation}\label{estm1}
	q\int_{B_{r}}M(\bar{v})^{q-1}m(|\nabla\bar{v}|)|\nabla\bar{v}|dx
	\leq\lambda\|\rho\|_{\infty}
	\int_{B_{r}}M(\bar{v})^{q-1}(\bar{v}+\|u\|_{\infty})m(\bar{v}+\|u\|_{\infty})dx.
	\end{equation}
	Let
	$$
	h(x)=\left\{
	\begin{array}{lcl}
	M(\bar{v}(x))^{q}&\mbox{ if }&x\in\Omega_{0},\\
	0&\mbox{ if }&x\notin \Omega_{0}.
	\end{array}
	\right.
	$$
	Using the following inequality
	\begin{equation}\label{ab}
	am(b)\leq bm(b)+am(a),
	\end{equation}
	with $a=|\nabla\bar{v}|$ and $b=\bar{v}$, we obtain
	$$
	\begin{array}{clc}
	\int_{B_{r}}|\nabla h|dx&\leq
	q\int_{B_{r}}M(\bar{v})^{q-1}m(|\nabla\bar{v}|)
	|\nabla\bar{v}|dx
	+q\int_{B_{r}}M(\bar{v})^{q-1}\bar{v}m(\bar{v})dx&\\
	&\leq
	q\int_{B_{r}}M(\bar{v})^{q-1}m(|\nabla\bar{v}|)
	|\nabla\bar{v}|dx\\
	&
	+q\int_{B_{r}}M(\bar{v})^{q-1}(\bar{v}+\|u\|_{\infty})m(\bar{v}+\|u\|_{\infty})dx.
	\end{array}
	$$
	In view of \eqref{estm1}, we obtain
	$$
	\int_{B_{r}}|\nabla h|dx\leq
	C_{2}\int_{B_{r}}M(\bar{v})^{q}dx\leq C_2M(2\|u\|_{\infty}+1)^{q}|\Omega|,
	$$
	where $C_{2}=\frac{(q+\lambda\|\rho\|_\infty)(1+3\|u\|_{\infty})m(1+3\|u\|_{\infty})}{M(r)}$.
	Therefore, $h\in W^{1,1}_{0}(B_{r})$ and so we can write
	$$
	\Big(\int_{B_{r}}M(\bar{v})^{\frac{qN}{N-1}}dx\Big)^{\frac{N-1}{N}}\leq C_{2}C(N)
	\int_{B_{r}}M(\bar{v})^{q}dx,
	$$
	where $C(N)$ stands for the constant in the continuous embedding
	$W^{1,1}_{0}(B_{r})\hookrightarrow L^{\frac{N}{N-1}}(B_{r})$. 
	Then, applying Lemma \ref{lem.Adx2} with $g=M(\bar{v})$ and $\alpha=1$ we obtain for any $p>0$
	$$
	\sup_{B_{r}}M(\bar{v})\leq C_{3}\Big[r^{-N}\int_{B_{r}}
	M(\bar{v})^{p}dx\Big]^{\frac{1}{p}},
	$$
	where $C_{3}=(C_{2}C(N))^{\frac{N}{p}}$. Hence, follows
	$$
	\sup_{B_{\frac{r}{2}}}M(\bar{v})\leq C_3\Big[r^{-N}\int_{B_{r}}
	M(\bar{v})^{p}dx\Big]^{\frac{1}{p}}.
	$$
	Since $\frac{t}{2}m(\frac{t}{2})\leq M(t)\leq tm(t)$ and $\bar{v}=v+r=u-\inf_{B_r}u +r$ we have
	$\sup_{B_{\frac{r}{2}}}M(\bar{v})\geq m(\frac{r}{2})\sup_{B_{\frac{r}{2}}}\frac{\bar{v}}{2}$ and
	$M(\bar{v})\leq\bar{v}m(1+2\|u\|_\infty)$, which yields
	$$
	\sup_{B_{\frac{r}{2}}}\bar{v}\leq C\Big[r^{-N}\int_{B_{r}}
	\bar{v}^{p}dx\Big]^{\frac{1}{p}},
	$$
	where $C=(C_{2}C(N))^{\frac{N}{p}}\frac{2m(1+2\|u\|_\infty)}{m(\frac{r}{2})}$. Hence, the inequality (\ref{sup1}) is proved.
\end{proof}
\begin{lemma}
	Let $B_{r}\subset\Omega$ be an open ball of radius $0<r\leq1$. Then, there exist two constants $C>0$ and $p_{0}>0$ such that
	\begin{equation}\label{inf1}
	\Big(r^{-N}\int_{B_{r}}v^{p_{0}}dx\Big)^{\frac{1}{p_{0}}}\leq C\Big(\inf_{B_{\frac{r}{2}}}v+r\Big),
	\end{equation}
	where $B_{\frac{r}{2}}$ is the ball of radius $r/2$ concentric with $B_{r}$.
\end{lemma}
\begin{proof}
	Let $\Omega_{0}$ be a compact subset of $\Omega$ such that $B_{\frac{r}{2}}\subset\Omega_{0}\subset B_{r}$. Let $q>1$ and let $\psi$ be the function defined by
	$$
	\psi(x)=\left\{
	\begin{array}{lcl}
	M(\bar{v}(x))^{-q-1}\bar{v}(x)&\mbox{ if }&x\in \Omega_{0},\\
	0&\mbox{ if }&x\notin \Omega_{0},
	\end{array}
	\right.
	$$
	where $\bar{v}=v+r$. On $\Omega_0$ we compute
	$$
	\nabla\psi=M(\bar{v})^{-q-1}\nabla\bar{v}+(-q-1)M(\bar{v})^{-q-2}m(\bar{v})\bar{v}\nabla\bar{v}.
	$$
	By Theorem \ref{thm.Adx1} we have $\psi\in W^{1}_{0}L_{M}(\Omega)$.
	Thus, using the function $\psi$ in \eqref{4.1} we obtain
	$$
	\begin{array}{clc}
	&\lambda\int_{B_{r}}\rho(x)M(\bar{v})^{-q-1}\bar{v}m(v+\inf_{B_r}u)dx\\
	&=\int_{B_{r}}M(\bar{u})^{-q-1}|\nabla\bar{v}|m(|\nabla\bar{v}|)dx\\
	&+(-q-1)\int_{B_{r}}M(\bar{v})^{-q-2}m(\bar{v})\bar{v}|\nabla\bar{v}|m(|\nabla\bar{v}|)dx.
	\end{array}
	$$
	By the fact that $\bar{v}m(\bar{v})\geq M(\bar{v})$, we get
	$$
	\lambda\int_{B_{r}}\rho(x)M(\bar{v})^{-q-1}\bar{v}m(v+\inf_{B_r}u)dx
	\leq -q\int_{B_{r}}M(\bar{v})^{-q-1}|\nabla\bar{v}|m(|\nabla\bar{v}|)dx.
	$$
	Thus, since on $B_{r}$ one has $|v+\inf_{B_r}u|\leq\bar{v}+\|u\|_\infty$ we obtain
	\begin{equation}\label{estm4}
	q\int_{B_{r}}M(\bar{v})^{-q-1}|\nabla\bar{v}|m(|\nabla\bar{v}|)dx\leq\lambda\|\rho\|_{\infty}\int_{B_{r}}M(\bar{v})^{-q-1}
	m(\bar{v}+\|u\|_\infty)(\bar{v}+\|u\|_\infty)dx.
	\end{equation}
	On the other hand, let $h$ be the function defined  by
	$$
	h(x)=\left\{
	\begin{array}{lcl}
	M(\bar{v}(x))^{-q}&\mbox{ if }&x\in\Omega_{0},\\
	0&\mbox{ if }&x\notin \Omega_{0}.
	\end{array}
	\right.
	$$
	Using once again \eqref{ab} with $a=|\nabla\bar{v}|$ and $b=\bar{v}$, we obtain
	$$
	\begin{array}{clc}
	\int_{B_{r}}|\nabla h|dx&\leq
	q\int_{B_{r}}M(\bar{v})^{-q-1}m(|\nabla\bar{v}|)
	|\nabla\bar{v}|dx
	+q\int_{B_{r}}M(\bar{v})^{-q-1}\bar{v}m(\bar{v})dx&\\
	&\leq
	q\int_{B_{r}}M(\bar{v})^{-q-1}m(|\nabla\bar{v}|)
	|\nabla\bar{v}|dx\\
	&
	+q\int_{B_{r}}M(\bar{v})^{q-1}(\bar{v}+\|u\|_{\infty})m(\bar{v}+\|u\|_{\infty})dx,
	\end{array}
	$$
	which together with \eqref{estm4} yield
	$$
	\int_{B_{r}}|\nabla h|dx\leq C_{2}
	\int_{B_{r}}M(\bar{v})^{-q}dx\leq C_2M(r)^{-q}|\Omega|,
	$$
	with $C_{2}=\frac{(q+\lambda\|\rho\|_{\infty})m(1+3\|u\|_{\infty})(1+3\|u\|_{\infty})}{M(r)}$. Thus, $h\in W^{1,1}_{0}(B_{r})$ and so
	we can write
	$$
	\Big(\int_{B_{r}}M(\bar{v})^{-\frac{qN}{N-1}}dx\Big)^{\frac{N-1}{N}}\leq C_{2}C(N)
	\int_{B_{r}}M(\bar{v})^{-q}dx,
	$$
	where $C(N)$ is the constant in the continuous embedding $W^{1,1}_{0}(B_{r})\hookrightarrow L^{\frac{N}{N-1}}(B_{r})$. Therefore, applying Lemma \ref{lem.Adx2} with $g=M(\bar{v})$ and $\alpha=-1$ we get for any $p>0$
	$$
	\sup_{B_r}M(\bar{v})^{-1}\leq(C_{2}C(N))^{\frac{N}{p}}\Big(r^{-N}\int_{B_{r}}M(\bar{v})^{-p}dx\Big)^{\frac{1}{p}}.
	$$
	So that one has
	$$
	\begin{array}{ccl}
	\Big(r^{-N}\int_{B_{r}}M(\bar{v})^{-p}dx\Big)^{\frac{-1}{p}}&\leq (C_{2}C(N))^{\frac{N}{p}}\inf_{B_{r}}M(\bar{v})\\
	&\leq (C_{2}C(N))^{\frac{N}{p}}\inf_{B_{\frac{r}{2}}}M(\bar{v}).
	\end{array}
	$$
	The fact that $M(\bar{v})\geq m\Big(\frac{r}{2}\Big)\frac{\bar{v}}{2}$ and $M(\bar{v})\leq m(2\|u\|_\infty+1)\bar{v}$, yields
	\begin{equation}\label{estm3.v}
	\Big(r^{-N}\int_{B_{r}}\bar{v}^{-p}dx\Big)^{\frac{-1}{p}}\leq C\inf_{B_{\frac{r}{2}}}\bar{v},
	\end{equation}
	where $C=(C_{2}C(N))^{\frac{N}{p}}\frac{2m(2\|u\|_\infty+1)}{m(\frac{r}{2})}$.
	Now, it only remains to show that there exist two constants $c>0$ and $p_{0}>0$, such that
	$$
	\Big(r^{-N}\int_{B_{r}}\bar{v}^{p_{0}}dx\Big)^{\frac{1}{p_{0}}}\leq c\Big(r^{-N}\int_{B_{r}}\bar{v}^{-p_{0}}dx\Big)
	^{\frac{-1}{p_{0}}}.
	$$
	Let $B_{r_1}\subset B_r$ and let $\Omega_{0}$ be a compact subset of $\Omega$ such that $B_{\frac{r_1}{2}}\subset\Omega_{0}\subset B_{r_1}$. Let $\psi$ be the function defined by
	$$
	\psi(x)=\left\{
	\begin{array}{lcl}
	\bar{v}(x)&\mbox{ if }&x\in\Omega_{0},\\
	0&\mbox{ if }&x\notin \Omega_{0}.
	\end{array}
	\right.
	$$
	Then, inserting $\psi$ as a test function in \eqref{4.1} we obtain
	$$
	\begin{array}{clc}
	\int_{B_{r_1}}m(|\nabla\bar{v}|)|\nabla\bar{v}|dx
	&\leq \lambda\|\rho\|_\infty\int_{B_{r_1}}m(|v+\inf_{B_R}u|)\bar{v}dx&\\
	&\leq \lambda\|\rho\|_\infty\int_{B_{r_1}}m(\bar{v}+\|u\|_\infty)(\bar{v}+\|u\|_\infty)dx.
	\end{array}
	$$
	Since $\bar{v}\leq(2\|u\|_\infty+1)$ and $|B_{r_1}|=r_1^N|B_1|$ we obtain
	\begin{equation}\label{inf.v}
	\int_{B_{r_1}}m(|\nabla\bar{v}|)|\nabla\bar{v}|dx\leq c_0r_{1}^{N},
	\end{equation}
	where $c_0=\lambda\|\rho\|_\infty m(3\|u\|_\infty+1)(3\|u\|_\infty+1)|B_1|$.
	On the other hand, we can use \eqref{ab} with $a=|\nabla\bar{v}|$ and $b=\frac{\bar{v}}{r_1}$ obtaining
	$$
	|\nabla\bar{v}|m\Big(\frac{\bar{v}}{r_1}\Big)\leq |\nabla\bar{v}|m(|\nabla\bar{v}|)+
	\frac{\bar{v}}{r_1}m\Big(\frac{\bar{v}}{r_1}\Big).
	$$
	Pointing out that $\frac{\bar{v}}{r_1}m(\frac{\bar{v}}{r_1})\geq M(\frac{\bar{v}}{r_1})\geq M(\frac{\bar{v}}{r})\geq M(1)$, we get
	$$
	\frac{|\nabla\bar{v}|}{\bar{v}}\leq \frac{1}{r_1 M(1)}m(|\nabla\bar{v}|)|\nabla\bar{v}|+\frac{1}{r_1}.
	$$
	Integrating over the ball $B_{\frac{r_1}{2}}$ and using \eqref{inf.v} we obtain
	$$
	\begin{array}{lll}
	\int_{B_{\frac{r_1}{2}}}\frac{|\nabla\bar{v}|}{\bar{v}}dx&\leq \frac{1}{r_1 M(1)}\int_{B_{\frac{r_1}{2}}}
	m(|\nabla\bar{v}|)|\nabla\bar{v}|dx +\frac{1}{r_1}|B_{\frac{r_1}{2}}|\\
	&\leq\bigg(\frac{c_0}{M(1)}+\frac{|B_1|}{2^N}\bigg)r_1^{N-1}.
	\end{array}
	$$
	The above inequality together with \cite[Lemma 1.2]{Trudinger67} ensure the existence of two constants $p_{0}>0$ and $c>0$ such that
	\begin{equation}\label{inf3.v}
	\Big(\int_{B_{r}}\bar{v}^{p_{0}}dx\Big)\Big(\int_{B_{r}}\bar{v}^{-p_{0}}dx\Big)
	\leq cr^{2N}.
	\end{equation}
	Finally, the estimate \eqref{inf1} follows from \eqref{estm3.v} with $p=p_{0}$ and \eqref{inf3.v}. 
\end{proof}
\begin{theorem}[Harnack-type inequalities] 
	Let $u\in W^{1}_{0}L_{M}(\Omega)$ be a bounded weak solution of \eqref{I} and let $B_{\frac{r}{2}}$, $0<r\leq1$, be a ball with  radius $\frac{r}{2}$. There exists a large constant $C>0$ such that 
	\begin{equation}\label{Harnack-v}
	\sup_{B_{\frac{r}{2}}}v\leq C(\inf_{B_{\frac{r}{2}}}v+r)
	\end{equation}
	and 
	\begin{equation}\label{Harnack-w}
	\sup_{B_{\frac{r}{2}}}w\leq C(\inf_{B_{\frac{r}{2}}}w+r).
	\end{equation}
\end{theorem}
\begin{proof}
	Putting together \eqref{sup1}, with the choice $p=p_{0}$, and \eqref{inf1} we immediately get \eqref{Harnack-v}. In the same way as above, one can obtain analogous inequalities to (\ref{inf1}) and (\ref{sup1}) for $w$ obtaining the inequality \eqref{Harnack-w}.
\end{proof}
\par We are now ready to prove the following H\"older estimate for weak solutions of \eqref{I}.
\begin{theorem}[H\"older regularity]\label{Holder}
	Let $u\in W^{1}_{0}L_{M}(\Omega)$ be a bounded weak solution of \eqref{I}. Then there exist two constants $0<\alpha<1$ and $C>0$ such that if $B_r$ and $B_R$ are two concentric balls of radii $0<r\leq R\leq1$, then
	$$
	\mbox{osc}_{B_{r}}u\leq C\Big(\frac{r}{R}\Big)^{\alpha}\Big(\sup_{B_{R}}|u|+C(R)\Big),
	$$
	where $\mbox{osc}_{B_{r}}u=\sup_{B_r}u-\inf_{B_{r}}u$ and $C(R)$ is a positive constant which depends on $R.$
\end{theorem}
\begin{proof} 
	From \eqref{Harnack-v} and \eqref{Harnack-w} we obtain
	$$
	\sup_{B_{\frac{r}{2}}}u-\inf_{B_{r}}u=\sup_{B_{\frac{r}{2}}}v\leq C(\inf_{B_{\frac{r}{2}}}v+r)=
	C(\inf_{B_{\frac{r}{2}}}u-\inf_{B_r}u+r)
	$$
	and
	$$
	\begin{array}{clc}
	\sup_{B_{r}}u-\sup_{B_{\frac{r}{2}}}u=\sup_{B_{\frac{r}{2}}}w\leq C(\inf_{B_{\frac{r}{2}}}w+r)
	\leq C(\sup_{B_{\frac{r}{2}}}w+r)
	=C(\sup_{B_{r}}u-\sup_{B_{\frac{r}{2}}}u+r).
	\end{array}
	$$
	Hence, summing up both the two first terms in the left-hand side and the two last terms in the right-hand side of the above inequalities,  we obtain
	$$
	\sup_{B_{r}}u-\inf_{B_{r}}u\leq
	C\Big(\sup_{B_{r}}u-\inf_{B_{r}}u+\inf_{B_{\frac{r}{2}}}u-\sup_{B_{\frac{r}{2}}}u+2r\Big),
	$$
	that is to say, what one still writes
	\begin{equation}\label{osc2}
	\mbox{osc}_{B_{\frac{r}{2}}}u\leq\Big(\frac{C-1}{C}\Big)\mbox{osc}_{B_{r}}u+2r.
	\end{equation}
	Let us fix  some real number $R_{1}\leq R$ and define
	$\sigma(r)=\mbox{osc}_{B_{r}}u$. Let $n\in\mathbb{N}$ be an integer. Iterating the inequality \eqref{osc2} by substituting $r=R_{1}$, $r=\frac{R_{1}}{2}$, $\cdots$, $r=\frac{R_{1}}{2^n}$,  we obtain
	$$
	\begin{array}{clc}
	\sigma\Big(\frac{R_{1}}{2^n}\Big)&\leq \gamma^{n}\sigma(R_1)+R_1\sum_{i=0}^{n-1}\frac{\gamma^{n-1-i}}{2^{i-1}}\\
	&\leq \gamma^{n}\sigma(R)+\frac{R_1}{1-\gamma},
	\end{array}
	$$
	where $\gamma=\frac{C-1}{C}.$ For any $r\leq R_1$, there exists an integer $n$ satisfying
	$$
	2^{-n-1}R_1\leq r<2^{-n}R_1.
	$$
	Since $\sigma$ is an increasing function, we get
	$$
	\sigma(r)\leq\gamma^{n}\sigma(R)+\frac{R_1}{1-\gamma}.
	$$
	Being $\gamma<1$, we can write
	$$
	\begin{array}{clc}
	\gamma^{n}\leq\gamma^{-1}\gamma^{-\frac{\log(\frac{r}{R_1})}{\log2}}=\gamma^{-1}\Big(\frac{r}{R_1}\Big)^{-\frac{\log\gamma}{\log2}}.
	\end{array}
	$$
	Therefore,
	$$
	\sigma(r)\leq\gamma^{-1}\Big(\frac{r}{R_1}\Big)^{-\frac{\log\gamma}{\log2}}\sigma(R)+\frac{R_1}{1-\gamma}.
	$$
	This inequality holds for arbitrary $R_1$ such that $r\leq R_1\leq R$.
	Let now $\alpha\in(0,1)$ and $R_1=R^{1-\alpha}r^{\alpha}$, so that we have from the preceding
	$$
	\sigma(r)\leq\gamma^{-1}\Big(\frac{r}{R}\Big)^{-(1-\alpha)\frac{\log\gamma}{\log2}}\sigma(R)+\frac{R}{1-\gamma}\Big(\frac{r}{R}\Big)^{\alpha}.
	$$
	Thus, the desired result follows by choosing $\alpha$ such that $\alpha=-(1-\alpha)\frac{\log\gamma}{\log2}$, that is $\alpha=\frac{-\log\gamma}{\log2-\log\gamma}$.
\end{proof}



\begin{thebibliography}{99}
	\bibitem{Adams}
	\newblock R. A. Adams and J. J. F. Fournier,
	\newblock Sobolev spaces,
	\newblock Pure and Applied Mathematics (Amsterdam).
	\newblock 140
	\newblock Second.
	\newblock Elsevier/Academic Press, Amsterdam,
	\newblock 2003.
	
	
	\bibitem{Anane} 
	\newblock A. Anane,
	\newblock Simplicit\'e et isolation de la premi\`ere valeur propre du
	{$p$}-laplacien avec poids,
	\newblock C. R. Acad. Sci. Paris S\'er. I Math. S\'erie
	I. Math\'ematique,
	\newblock 305,
	\newblock 1987.
	\newblock 16,
	\newblock 725--728.
	
	
	\bibitem{Brezis}
	\newblock H. Brezis,
	\newblock Functional analysis, {S}obolev spaces and partial differential
	equations,
	\newblock Universitext,
	\newblock Springer, New York,
	\newblock 2011.
	
	
	
	\bibitem{Ekeland}  
	\newblock I. Ekeland and R. Temam,
	\newblock Convex analysis and variational problems,
	\newblock Translated from the French,
	Studies in Mathematics and its Applications, Vol. 1,
	\newblock North-Holland Publishing Co., Amsterdam-Oxford; American
	Elsevier Publishing Co., Inc., New York,
	\newblock 1976.
	
	
	
	\bibitem{Garcia}   
	\newblock M. Garc\'ia-Huidobro and V. K. Le and R. Man\'asevich and
	K. Schmitt,
	\newblock On principal eigenvalues for quasilinear elliptic differential
	operators: an {O}rlicz-{S}obolev space setting,
	\newblock NoDEA Nonlinear Differential Equations Appl.
	\newblock 6,
	\newblock 1999.
	\newblock 2,
	\newblock 207--225.
	
	
	
	\bibitem{Gossez} 
	\newblock J. P. Gossez,
	\newblock Nonlinear elliptic boundary value problems for equations with
	rapidly (or slowly) increasing coefficients,
	\newblock Trans. Amer. Math. Soc,
	\newblock 190,
	\newblock 1974,
	\newblock 163--205.
	
	
	
	\bibitem{Gossez82} 
	\newblock J. P. Gossez,
	\newblock Some approximation properties in {O}rlicz-{S}obolev spaces,
	\newblock  Studia Math,
	\newblock  74,
	\newblock 1982,
	\newblock 1,
	\newblock 17--24.
	
	
	
	\bibitem{Gossez86} 
	\newblock J. P. Gossez,
	\newblock A strongly nonlinear elliptic problem in {O}rlicz-{S}obolev
	spaces,
	\newblock Nonlinear functional analysis and its applications, {P}art 1
	({B}erkeley, {C}alif., 1983),
	\newblock Proc. Sympos. Pure Math,
	\newblock 45,
	\newblock 455--462,
	\newblock Amer. Math. Soc., Providence, RI,
	\newblock 1986.
	
	
	
	\bibitem{Manasevich} 
	\newblock J. P. Gossez and R. Man\'asevich,
	\newblock On a nonlinear eigenvalue problem in {O}rlicz-{S}obolev
	spaces,
	\newblock Proc. Roy. Soc. Edinburgh Sect. A,
	\newblock 132,
	\newblock 2002,
	\newblock 4,
	\newblock 891--909.
	
	
	
	\bibitem{Kras}  
	\newblock M. A. Krasnosel'skii and J. B. Rutic'kii,
	\newblock Convex functions and {O}rlicz spaces,
	\newblock Translated from the first Russian edition by Leo F. Boron,
	\newblock P. Noordhoff Ltd., Groningen,
	\newblock 1961.
	
	
	
	\bibitem{Lindqvist}  
	\newblock P. Lindqvist,
	\newblock On the equation ${\rm div}\,(|\nabla u|^{p-2}\nabla
	u)+\lambda|u|^{p-2}u=0$,
	\newblock Proc. Amer. Math. Soc,
	\newblock 109,
	\newblock 1990,
	\newblock 1,
	\newblock 157--164.
	
	
	
	\bibitem{Lorca}
	\newblock M. Montenegro and S. Lorca,
	\newblock The eigenvalue problem for quasilinear elliptic operators with
	general growth,
	\newblock Appl. Math. Lett,
	\newblock 25,
	\newblock 2012,
	\newblock 7,
	\newblock 1045--1049.
	
	
	
	
	
	\bibitem{MRR} 
	\newblock M. Mih\u{a}ilescu, V. R\u{a}dulescu and D. Repov\v{s},
	\newblock On a non-homogeneous eigenvalue problem involving a potential:
	an Orlicz-Sobolev space setting,
	\newblock J. Math. Pures Appl. (9),
	\newblock 93,
	\newblock 2010,
	\newblock 2,
	\newblock 132--148.
	
	
	\bibitem{MR} 
	\newblock M. Mih\u{a}ilescu, and V. R\u{a}dulescu,
	\newblock Eigenvalue problems associated with nonhomogeneous
	differential operators in {O}rlicz-{S}obolev spaces,
	\newblock Anal. Appl. (Singap.),
	\newblock 6,
	\newblock 2008,
	\newblock 1,
	\newblock 83--98.
	
	
	\bibitem{Mustonen} 
	\newblock V. Mustonen and M. Tienari,
	\newblock An eigenvalue problem for generalized {L}aplacian in
	{O}rlicz-{S}obolev spaces,
	\newblock Proc. Roy. Soc. Edinburgh Sect. A,
	\newblock 129,
	\newblock 1999,
	\newblock 1,
	\newblock 153--163.
	
	
	
	\bibitem{Otani} 
	\newblock M. \^Otani and T. Teshima,
	\newblock On the first eigenvalue of some quasilinear elliptic
	equations,
	\newblock Proc. Japan Acad. Ser. A Math. Sci.,
	\newblock 64,
	\newblock 1988,
	\newblock 1,
	\newblock 8--10,
	
	
	
	
	\bibitem{Pucci}
	\newblock P. Pucci and J. Serrin,
	\newblock A note on the strong maximum principle for elliptic
	differential inequalities,
	\newblock J. Math. Pures Appl. (9),
	\newblock 79,
	\newblock 2000,
	\newblock 1,
	\newblock 57--71.
	
	
	
	\bibitem{Sakaguchi}
	\newblock S. Sakaguchi,
	\newblock Concavity properties of solutions to some degenerate
	quasilinear elliptic {D}irichlet problems,
	\newblock Ann. Scuola Norm. Sup. Pisa Cl. Sci. (4),
	\newblock 14,
	\newblock 1987,
	\newblock 3,
	\newblock 403--421 (1988),
	
	
	
	
	\bibitem{stamp65}
	\newblock G. Stampacchia,
	\newblock Le probl\`eme de {D}irichlet pour les \'equations elliptiques du
	second ordre \`a coefficients discontinus,
	\newblock Ann. Inst. Fourier (Grenoble),
	\newblock 15,
	\newblock 1965,
	\newblock fasc. 1,
	\newblock 189--258,
	
	
	
	\bibitem{Trudinger67} 
	\newblock N. S. Trudinger,
	\newblock On {H}arnack type inequalities and their application to
	quasilinear elliptic equations,
	\newblock Comm. Pure Appl. Math.,
	\newblock 1967,
	\newblock 721--747.
	
\end{thebibliography}

\end{document}